\documentclass{article}
\usepackage{amsmath, amssymb, latexsym}
\usepackage{amsopn,amsfonts, amsthm}
\usepackage[T1]{fontenc}

\newtheorem{theorem}{Theorem}[section]

\newtheorem{corollary}[theorem]{Corollary}

\newtheorem{proposition}[theorem]{Proposition}
\newtheorem{lemma}[theorem]{Lemma}
\theoremstyle{definition}

\newtheorem{definition}[theorem]{Definition}

\newtheorem{problem}[theorem]{Problem}

\newtheorem{remark}[theorem]{Remark}

\begin{document}

\title{On Loeb and sequential spaces in $\mathbf{ZF}$}
\author{Kyriakos Keremedis and Eliza Wajch}
\maketitle

\begin{abstract}
A topological space is called Loeb if the collection of all its non-empty
closed sets has a choice function. In this article, in the absence of the
axiom of choice, connections between Loeb and sequential spaces are
investigated. The following main theorems are proved in $\mathbf{ZF}$%
:\smallskip\ 

(i) If $\mathbf{X}$ is a Cantor completely metrizable second-countable
space, then $\mathbf{X}^{\omega}$ is Loeb.\smallskip

(ii) If a sequential, sequentially locally compact space $\mathbf{X}$ has the property that every infinite countable collection of non-empty closed subsets of $\mathbf{X}$ has a choice function, the Cartesian product $\mathbf{X}\times\mathbf{Y}$ of $\mathbf{X}$ with any sequential space $\mathbf{Y}$ is sequential. The Cartesian product of
a sequential locally countably compact space with any sequential space is
sequential. \smallskip

(iii) If $\mathbf{X}$ is a first-countable regular Loeb space, then $\mathbf{%
X}$ is sequential if and only if, for every family $\{A_j: j\in J\}$ for
non-empty sequentially closed subsets of $\mathbf{X}$, there exists a family 
$\{B_j: j\in J\}$ of non-empty well-orderable sets such that $B_j\subseteq
A_j$ for each $j\in J$. In consequence, it is relatively consistent with $%
\mathbf{ZF}$ that $\mathbb{R}^{\omega}$ is sequential, while $\mathbf{CAC}(%
\mathbb{R})$ fails.\smallskip

(iv) If $\mathbb{R}$ is sequential, then every second-countable compact
Hausdorff space is sequential. In consequence, the Hilbert cube $%
[0,1]^{\omega}$ is sequential if and only if $\mathbb{R}$ is
sequential.\smallskip

It is also proved that, in some models of $\mathbf{ZF}$, a countable product
of Cantor completely metrizable second-countable spaces can fail to be Loeb
and it is independent of $\mathbf{ZF}$ that every sequential subspace of $%
\mathbb{R}$ is Loeb. Some other sentences are shown to be independent of $%
\mathbf{ZF}$. Several open problems are posed, among them, the following
question: is $\mathbb{R}^{\omega}$ sequential if $\mathbb{R}$ is
sequential? \bigskip

\end{abstract}

\section{Introduction}

In this paper, the intended context for reasoning and statements of theorems
is the Zermelo-Fraenkel set theory $\mathbf{ZF}$ without the axiom of choice 
$\mathbf{AC}$. The system $\mathbf{ZF+AC}$ is denoted by $\mathbf{ZFC}$. A
good introduction to $\mathbf{ZF}$ is given in \cite{ku}. Let us recall that
if $J$ is a non-empty set and $\mathcal{A}=\{A_j: j\in J\}$ is a collection of non-empty sets, then every
element of the product $\prod_{j\in J} A_j$ is called a \textit{choice function%
} of $\mathcal{A}$, while every choice function of $\{A_j: j\in J^{\star}\}$ for an infinite subset $J^{\star}$ of $J$ 
is called a \textit{partial choice function} of $\mathcal{A}$.

For a topological space $\mathbf{X}=\langle X, \tau\rangle $, let $Cl(%
\mathbf{X})$ be the collection of all closed sets of $\mathbf{X}$ and let $%
Lo(\mathbf{X})$ be the collection of all choice functions of $Cl(\mathbf{X}%
)\setminus\{\emptyset\}$. The space $\mathbf{X}$ is said to be \textit{Loeb}
if either $X=\emptyset$ or $Lo(\mathbf{X})\neq\emptyset$ (see \cite{kertach}%
). If $\mathbf{X}$ is a non-empty Loeb space, then every function from $Lo(%
\mathbf{X})$ is called a \textit{Loeb function of} $\mathbf{X}$. As usual,
given a topological property $\mathcal{T}$, we say that a metric space $%
\langle Y, d\rangle$ has $\mathcal{T}$ if the topological space $\langle Y,
\tau(d)\rangle$ has $\mathcal{T}$ where $\tau(d)$ is the topology induced by
the metric $d$. In particular, a metric space $\langle Y, d\rangle$ is
called \textit{Loeb} if the topological space $\langle Y, \tau(d)\rangle$ is
Loeb.

In \cite{loeb}, P. A. Loeb proved in $\mathbf{ZF}$ that if $\kappa $ is an
infinite cardinal number of von Neumann (that is, $\kappa $ is an initial
ordinal number of von Neumann), while $\{\mathbf{X}_{i}:i\in \kappa \}$ is a
family of compact spaces such that the collection $\{Lo(\mathbf{X}_{i}):i\in
\kappa \}$ has a choice function, then the Tychonoff product $\mathbf{X}%
=\prod_{i\in \kappa }\mathbf{X}_{i}$ is compact. The significance of Loeb's
result for metric spaces is that one does not need any form of the axiom of
choice to prove that countable products of second-countable compact metric
spaces are compact. In particular, the Hilbert cube $[0,1]^{\omega }$ and
the Cantor cube $2^{\omega }$ are compact in $\mathbf{ZF}$. Newer results
concerning Loeb spaces were obtained in \cite{kyr1} and \cite{kertach}.

We recall that a subset $A$ of a topological space $\mathbf{X}$ is called 
\textit{sequentially closed} in $\mathbf{X}$ if all limits of convergent in $%
\mathbf{X}$ sequences of points of $A$ belong to $A$. We denote by $SC(%
\mathbf{X})$ the collection of all sequentially closed sets of $\mathbf{X}$.
A topological space $\mathbf{X}$ is called \textit{sequential} if $SC(%
\mathbf{X})=Cl(\mathbf{X})$.

A natural generalization of the notion of a Loeb space is the following new
concept:

\begin{definition}
\label{d1.1} (i) A topological space $\mathbf{X}$ will be called \textit{%
s-Loeb} if the collection $SC(\mathbf{X})\backslash \{\emptyset \}$ has a
choice function. \newline
(ii) If $\mathbf{X}$ is an s-Loeb space, then every choice function of $SC(%
\mathbf{X})$ will be called an \textit{s-Loeb} function of $\mathbf{X}$.
\end{definition}

Of course, every sequential s-Loeb space is Loeb. However, a Loeb space need
not be sequential. In this paper, we show a number of sufficient conditions
on a Loeb space to be sequential. To this aim, in Section 2, we prove that
if $\mathbf{X}$ is a Cantor completely metrizable second-countable space,
then $\mathbf{X}^{\omega }$ is Loeb. In particular, $\mathbb{R}^{\omega }$, the Baire space $\mathbb{N}^{\omega}$
and the Hilbert cube $[0,1]^{\omega }$ are Loeb spaces. In Section 2, we
also prove that it is independent of $\mathbf{ZF}$ that, for every
collection $\{\mathbf{X}_{n}:n\in\omega\}$ of Cantor completely metrizable
second-countable spaces, the product $\prod_{n\in\omega}\mathbf{X}_{n}$ is
Loeb. Section 3 is devoted to sequential spaces and their Cartesian
products. We give new $\mathbf{ZF}$-proofs to several famous theorems on
Cartesian products of sequential spaces because $\mathbf{AC}$ is involved in
their known proofs. We prove that if a sequential locally
sequentially compact space $\mathbf{X}$ is Loeb (or every infinite countable collection of non-empty closed sets of $\mathbf{X}$ has a partial choice function), then $\mathbf{X}\times\mathbf{Y}$ is
sequential for every sequential space $\mathbf{Y}$. We apply this result for
showing that it is provable in $\mathbf{ZF}$ that if $\mathbf{X}$ is a
locally countably compact sequential space, then $\mathbf{X}\times\mathbf{Y}$
is sequential for every sequential space $\mathbf{Y}$. We prove that if $%
\mathbf{X}$ and $\mathbf{Y}$ are sequential Loeb spaces such that $\mathbf{X}$ is 
$T_1$ and $\mathbf{Y}$ is a sequentially compact $T_3$-space, then $\mathbf{X%
}\times\mathbf{Y}$ is sequential and Loeb. In Section 4, among a
considerable number of other results, we show that if $\mathbb{R}$  is
sequential, then so is $\mathbf{X}^{\omega }$ for every second-countable
Hausdorff compact space $\mathbf{X}$. In particular, the Hilbert cube $%
[0,1]^{\omega }$ is sequential if and only if $\mathbb{R}$ is sequential.
This leads to distinct than in \cite{kw} ways of solutions to Problems 6.6
and 6.11 of \cite{OPWZ}. Finally, we prove that a Loeb regular
first-countable space $\mathbf{X}$ is sequential if and only if, for every
family $\{A_j: j\in J\}$ of non-empty sequentially closed subsets of $%
\mathbf{X}$, there exists a family $\{B_j: j\in J\}$ of non-empty
well-orderable sets such that $B_j\subseteq A_j$ for each $j\in J$. This
helps us to prove that it is independent of $\mathbf{ZF}$ that both $\mathbb{%
R}^{\omega}$ is sequential and there exists an infinite countable family of
non-empty subsets of $\mathbb{R}$ which does not have a choice function.

As in $\cite{ku}$, we denote by $\omega$ the set of all finite ordinal
numbers of von Neumann and, if $n\in\omega$, then $n+1=n\cup\{n\}$. We put $%
\mathbb{N}=\omega\setminus\{0\}$ where $0=\emptyset$. All ordinal numbers
considered in this paper are assumed to be von Neumann ordinals. For a set $X
$, let $\mathcal{P}(X)$ denote the power set of $X$.

Below, we list the choice principles we shall be dealing with in the sequel.
For the known forms which can be found in \cite{hr} or \cite{herl}, we give
the form number under which they are recorded in \cite{hr} or we refer to
their definition in \cite{herl}. The symbol $\mathbf{X}$ stands for a
topological or metric space in the forms defined below:

\begin{itemize}
\item $\mathbf{CAC}$ (Form 8 in \cite{hr}, Definition 2.5 in \cite{herl}) :
Every infinite countable family of non-empty sets has a choice function.

\item $\mathbf{PCAC}$ (Form [8B] in \cite{hr}, Definition 2.11(2) in \cite%
{herl}) : Every infinite countable family of non-empty sets has a partial
choice function.

\item $\mathbf{CAC}_{\omega}$ : Every infinite countable family of
non-empty countable sets has a choice function.

\item $\mathbf{CAC}(\mathbb{R})$ (Form 94 in \cite{hr}, Definition 2.9 (1)
in \cite{herl}) : $\mathbf{CAC}$ restricted to countable families of subsets
of $\mathbb{R}$.

\item $\mathbf{CAC}_{\omega }(\mathbb{R})$ : $\mathbf{CAC}$ restricted to
countable families of countable subsets of $\mathbb{R}$.

\item $\mathbf{CAC}_{D}(\mathbb{R})$ (\cite{kw}) : Every family $\mathcal{A}%
=\{A_{n}:n\in \omega\}$ of dense subsets of $\mathbb{R}$ has a partial
choice function.

\item $\mathbf{CAC}_{\omega D}(\mathbb{R})$ : Every family $\mathcal{A}%
=\{A_{n}:n\in \omega\}$ of countable dense subsets of $\mathbb{R}$ has a
partial choice function.

\item $\omega -\mathbf{CAC}(\mathbb{R})$ (Definition 4.56 in \cite{herl}) :
For every family $\mathcal{A}=\{A_{i}:i\in \omega \}$ of non-empty subsets
of $\mathbb{R}$, there exists a family $\mathcal{B}=\{B_{i}:i\in \omega \}$
of countable non-empty subsets of $\mathbb{R}$ such that, for each $i\in
\omega $, $B_{i}\subseteq A_{i}$.

\item $\mathbf{IDI}$ (Form 9 in \cite{hr}, Definition 2.13 in \cite{herl} ) :
Every infinite set is Dedekind-infinite.

\item $\mathbf{IWDI}$ (Form 82 in \cite{hr}) : For every infinite set $X$,
the power set $\mathcal{P}(X)$ is Dedekind-infinite. In other words, every
infinite set is weakly Dedekind-infinite.

\item $\mathbf{IDI}(\mathbb{R})$ (Form 13 in \cite{hr}, Definition 2.13(2)
in \cite{herl}) : Every infinite set of reals is Dedekind-infinite.

\item $\mathbf{BPI}$ (Form 14 in \cite{hr}, Definition 2.15 in \cite{herl}) : Every Boolean algebra has a prime ideal.

\item $\mathbf{L}(\mathbf{X})$ : The space $\mathbf{X}$ is Loeb.

\item $\mathbf{S}(\mathbf{X})$ : The space $\mathbf{X}$ is sequential.

\item $\mathbf{s-Loeb}(\mathbf{X})$ : The family of all non-empty
sequentially closed subsets of $\mathbf{X}$ has a choice function.

\item $\mathbf{CAC}(\mathbf{X},scl)$ : Every countable family of non-empty
sequentially closed subsets of $\mathbf{X}$ has a choice function.

\item $\mathbf{PCAC}(\mathbf{X},scl)$ : Every infinite countable family of
non-empty sequentially closed subsets of $\mathbf{X}$ has a partial choice
function.
\end{itemize}

If $\mathbf{X}=\langle X, d\rangle$ is a metric space, then, for $x\in X$
and a positive real number $\varepsilon$, let $B_d(x, \varepsilon)=\{ y\in
X: d(x, y)<\varepsilon\}$ and $\overline{B}_d(x, \varepsilon)=\{y\in X: d(x,
y)\leq\varepsilon\}$. For $A\subseteq X$, let $\delta_d(A)$ be the diameter
of $A$ in the metric space $\langle X, d\rangle$. As it has been said above,
the topology induced by the metric $d$ is denoted by $\tau(d)$. We recall
the following known definitions:

\begin{definition}
\label{d1.2} It is said that a metric $d$ on a set $X$ is \textit{Cantor
complete} if, for every sequence $(A_n)_{n\in\omega}$ of non-empty closed
subsets of $\langle X, d\rangle$ such that $\lim\limits_{n\to\infty}%
\delta_d(A_n)=0$ and $A_{n+1}\subseteq A_n $ for each $n\in\omega$, the
intersection $\bigcap_{n\in\omega}A_n$ is non-empty.
\end{definition}

\begin{definition}
\label{d1.3} A topological space $\langle X, \tau\rangle$ is called \textit{%
Cantor completely metrizable} if there exists a Cantor complete metric $d$
on $X$ such that $\tau(d)=\tau$.
\end{definition}

In the sequel, boldface letters will denote topological or metric spaces
(called spaces in abbreviation if this does not lead to misunderstanding),
while lightface letters will denote the underlying sets of the spaces under
consideration.

\begin{definition}
\label{d1.4} Let $\mathbf{X}$ be a topological space and let $A\subseteq X$.

\begin{enumerate}
\item[(i)] The \textit{sequential closure} of $A$ in $\mathbf{X}$ is the set 
$\text{scl}_{\mathbf{X}}(A)$ of all points of $X$ which are limits of
convergent in $\mathbf{X}$ sequences of points of $A$.

\item[(ii)] $A$ is called \textit{sequentially open} in $\mathbf{X}$ if $%
X\setminus A$ is sequentially closed in $\mathbf{X}$, that is, if $\text{scl}%
_{\mathbf{X}}(X\setminus A)\subseteq X\setminus A$.

\item[(iii)] The \textit{closure} of $A$ in $\mathbf{X}$ is denoted by $%
\text{cl}_\mathbf{X}(A)$.

\item[(iv)] $\mathbf{X}$ is called \textit{Fr\'echet-Urysohn} if $\text{cl}_{%
\mathbf{X}}(B)\subseteq\text{scl}_{\mathbf{X}}(B)$ for each $B\subseteq X$.

\item[(v)] $A$ is called \textit{sequentially compact} in $\mathbf{X}$ if
each sequence of points of $A$ has a subsequence convergent in $\mathbf{X}$
to a point from $A$.
\end{enumerate}
\end{definition}

Other topological notions used in this article, if they are not defined
here, are standard and can be found in \cite{En} or are slight modifications
of notions given in \cite{En}. For instance, contrary to \cite{En}, we omit
separation axioms in the definitions of compact and regular spaces. Our
compact spaces are called \textit{quasi-compact} in \cite{En}.

Given a set $S$, an element $t\in S$ and a family $\{X_{s}:s\in S\}$ of
sets, we denote by $\pi _{t}$ the projection of $\prod_{s\in S}X_{s}$ into $%
X_{t}$ defined by: $\pi _{t}(x)=x(t)$ for each $x\in \prod_{s\in S}X_{s}$.

We recall that if $(\mathbf{X}_n)_{n\in\omega}$ is a sequence of metric
spaces $\mathbf{X}_n=\langle X_n, d_n\rangle$ and, for all $n\in\omega$ and $%
a,b\in X_{n}$, $\rho _{n}(a,b)=\min \{1,d_{n}(a,b)\}$, then, for $%
X=\prod_{n\in\omega}X_n$, the function $d:X\times X\rightarrow \mathbb{R}$, 
given by: 
\begin{equation}
d(x,y)=\sum\limits_{n\in \omega}\frac{\rho_n (x(n),y(n))}{2^{n+1}}  \label{4}
\end{equation}%
for all $x,y\in X$, is a metric on $X$ such that the topology $\tau (d)$
coincides with the product topology of the family of topological spaces $%
\{\langle X_{n},\tau (d_{n})\rangle :n\in\omega\}$. In the sequel, we shall
always assume that whenever a family $\{\mathbf{X}_{n}: n\in\omega\}$ of
metric spaces $\mathbf{X}_n=\langle X_n, d_n\rangle$ is given, then $%
X=\prod_{n\in\omega}X_{n}$ and $\mathbf{X}=\langle X,d\rangle $ where $d$ is
the metric on $X$ defined by (\ref{4}).

If it is not otherwise stated, we consider $\mathbb{R}$ as a metric space
whose metric is induced by the standard absolute value of $\mathbb{R}$.
Moreover, for $n\in\mathbb{N}$, the space $\mathbb{R}^n$ is considered with
the Euclidean metric denoted by $\rho_{e,n}$ here.
As a topological space, $\mathbb{R}^n$ is considered with the natural
topology $\tau_{nat}$ induced by the metric $\rho_{e,n}$.  When $X\subseteq 
\mathbb{R}^n$, then the topological subspace $\mathbf{X}$ of $\mathbb{R}^n$
is the space $\langle X, \tau_{nat\vert_X}\rangle$ where $%
\tau_{nat\vert_X}=\{X\cap U: U\in\tau_{nat}\}$, while the metric subspace $%
\mathbf{X}$ of $\mathbb{R}^n$ is the space $\langle X, \rho_{e,n}\cap ((X\times
X)\times\mathbb{R})\rangle$.

\section{ Cantor completely metrizable and Loeb spaces}

In $\mathbf{ZFC}$, every topological space is Loeb. However, there exist $%
\mathbf{ZF}$-models in which there are non-Loeb metrizable spaces. For
instance, in Cohen's Basic Model $\mathcal{M}1$ in \cite{hr}, the subspace $%
\mathbf{A}$ of all added Cohen reals of $\mathbb{R}$ is not Loeb. Indeed, $A$
is a dense in $\mathbb{R}$ infinite Dedekind-finite set, so the family $%
\{[p,q]\cap A:p,q\in \mathbb{Q},p<q\}$ is a family of non-empty closed
subsets of $\mathbf{A}$ which does not have any choice function. As it was
pointed out in \cite{kyr1}, all well-orderable metrizable spaces, the real
line $\mathbb{R}$ with the usual topology $\tau_{nat}$, all compact metrizable spaces
which have well-orderable bases (cf. \cite{kt}) are among $\mathbf{ZF}$-examples of Loeb spaces. The facts stated in the following proposition were
given in \cite{kyr1}:

\begin{proposition}
\label{p2.1} (i) Every closed subspace of a Loeb space is Loeb.\newline
(ii) Continuous images of Loeb spaces are Loeb.\newline
(ii) $\mathbf{BPI}$ implies that every compact Hausdorff space is Loeb.
\end{proposition}

It was proved in \cite{kyr1} that a product of Loeb Hausdorff spaces may
fail to be Loeb in a model of $\mathbf{ZF}$. However, we can prove in $%
\mathbf{ZF}$ the following theorem:

\begin{theorem}
\label{t2.2} If $\mathbf{X}$ is a Cantor completely metrizable second-countable space then $\mathbf{X}^{\omega }$ is Loeb.
\end{theorem}

\begin{proof} Let us fix a countable base $\mathcal{B}$ of a second-countable Cantor completely metrizable space $\mathbf{X}$. For $n\in\mathbb{N}$, let $\mathcal{A}_n=\{U\times X^{\omega\setminus n}: U\in\mathcal{B}^{n}\}$. Since there is a sequence $(\psi_n)_{n\in\mathbb{N}}$ of injections $\psi_n:\mathcal{A}_n\to\omega$, the set $\mathcal{A}=\bigcup_{n\in\mathbb{N}}\mathcal{A}_n$ is countable. Clearly, $\mathcal{A}$ can serve as an open base of $\mathbf{X}^{\omega}$, so the space $\mathbf{X}^{\omega}$ is second-countable. In \cite{kyr2}, it was shown in $\mathbf{ZF}$ that countable products of Cantor complete metric spaces are Cantor complete metric spaces. This implies that $\mathbf{X}^{\omega}$ is Cantor completely metrizable. Let $d$ be any Cantor complete metric on $X^{\omega}$ which induces the topology of the product $\mathbf{X}^{\omega}$. 
Let $C$ be a non-empty closed subset of $\mathbf{X}^{\omega}$. We show below how to choose an element of $C$. 

Since $\mathcal{A}$ is countable, it is well-orderable. In what follows, we refer to a fixed well-ordering of $\mathcal{A}$. Let $A_1(C)$ be the first element of $\mathcal{A}$ such that $C\cap A_1(C)\neq\emptyset$ and $\delta_d(A_1(C))< 1$. Suppose that, for $n\in\mathbb{N}$, we have already defined $A_n(C)$ such that $C\cap A_n(C)\neq\emptyset$. Let $A_{n+1}(C)$ be the first element of $\mathcal{A}$ such that $\text{cl}_{\mathbf{X}^{\omega}}(A_{n+1}(C))\subseteq A_{n}(C)$, $C\cap A_{n+1}(C)\neq\emptyset$ and $\delta_d(A_{n+1}(C))<\frac{1}{n+1}$. Since $d$ is Cantor complete, the sequence $(A_n(C))_{n\in\mathbb{N}}$, defined above by induction, has the property that the set $C\cap\bigcap_{n\in\mathbb{N}} A_n(C)$ is a singleton. This proves that $\mathbf{X}^{\omega}$ is a Loeb space.
\end{proof}

\begin{corollary}
\label{c2.3} (i) Every closed subspace of $\mathbb{R}^{\omega }$ is Loeb and
every continuous image of $\mathbb{R}^{\omega }$ is Loeb. In particular, $%
\mathbb{R}^{\omega }$, $(0,1)^{\omega }$, the Cantor cube $\{0,1\}^{\omega }$%
, the Hilbert cube $[0,1]^{\omega }$ and the Baire space $\mathbb{N}^{\omega
}$, hence the space of irrationals also, are Loeb.
\end{corollary}

\begin{remark}
\label{r2.4} In Theorem \ref{t2.2}, instead of the second countability of $
\mathbf{X}$, we can assume that $\mathbf{X}$ has a well-orderable base. In
particular, if $\mathbf{X}$ is a discrete metric space with a well-orderable
underlying set, then $\mathbf{X}^{\omega }$ is a Loeb space (see, for
instance, \cite{keta}).
\end{remark}

\begin{remark}
\label{r2.5} Suppose that $X$ is a non-empty set such that the collection $%
\mathcal{P}(X)\setminus\{\emptyset\}$ has a choice function $f$. As in
standard proofs that $\mathbf{AC}$ implies that every set is well-orderable,
using $f$, one can easily define, via a straightforward transfinite
induction on the ordinals, a well-ordering of $X$. Let us briefly recall how
to fix in $\mathbf{ZF}$ a well-ordering of $X$. We put $x_{0}=f(X)$ and
suppose that $\alpha$ is an ordinal such that a set $X(\alpha)=\{x_{\gamma}:
\gamma\in\alpha+1\}$ of pairwise distinct elements of $X$ has been defined.
Of course, if $X(\alpha)=X$, then we can fix a well-ordering of $X$. If $%
X(\alpha)\neq X$, we define $x_{\alpha+1}=f(X\setminus X(\alpha))$ and $%
X(\alpha+1)=X(\alpha)\cup\{x_{\alpha+1}\}$. Finally, suppose that $\beta$ is
a limit ordinal such that the sets $X(\alpha)=\{x_{\gamma}:
\gamma\in\alpha+1\}$ of pairwise distinct elements of $X$ have been defined
for each $\alpha\in\beta$. If $X\neq\bigcup_{\alpha\in\beta} X(\alpha)$, we
put $x_{\beta}=f(X\setminus\bigcup_{\alpha\in\beta}X(\alpha))$ and $%
X(\beta)=\{x_{\alpha}: \alpha\in\beta+1\}$ Otherwise, if $X=\{x_{\alpha}:
\alpha\in\beta\}$, we can fix a well-ordering of $X$. There must exist an
ordinal $\beta$ such that either $X=X(\beta)$ or $X=\bigcup_{\alpha\in\beta}
X(\alpha)$.
\end{remark}

\begin{theorem}
\label{t2.6} (a) $\mathbf{AC}$ if and only if every Cantor completely
metrizable space is Loeb.\newline
(b) If, for every sequence $(\mathbf{X}_n)_{n\in \omega}$ of simultaneously
compact and metrizable second-countable spaces, the product $\mathbf{X}%
=\prod_{n\in\omega}\mathbf{X}_{n}$ is metrizable and Loeb, then $\mathbf{CAC}%
_{\omega}$ holds.\newline
(c) If, for every sequence $(\mathbf{X}_n)_{n\in\omega}$ of Cantor
completely metrizable second-countable spaces, the product $\mathbf{X}%
=\prod_{n\in\omega}\mathbf{X}_{n}$ is Loeb, then all countable unions of
countable sets are well-orderable.\newline
(d) $\mathbf{CAC}$ implies that, for every sequence $(\mathbf{X}%
_n)_{n\in\omega}$ of Cantor completely metrizable second-countable spaces,
the product $\mathbf{X}=\prod_{n\in\omega}\mathbf{X}_{n}$ is Cantor
completely metrizable and second-countable, so Loeb.
\end{theorem}

\begin{proof}
(a) ($\rightarrow $) \medskip This is straightforward.

($\leftarrow $) Let $\mathcal{E}$ be a non-empty collection of non-empty sets. Let $X=\bigcup\mathcal{E}$. We notice that the discrete space $\mathbf{X}=\langle X, \mathcal{P}(X)\rangle$ is Cantor completely metrizable. Of course, if $\mathbf{X}$ is Loeb, then $\mathcal{E}$ has a choice function.

(b) Fix a pairwise disjoint family $\mathcal{A}=\{A_{n}: n\in\omega\}$
of non-empty countable sets.  Let $\infty$ be a given element which does not belong to $\bigcup_{n\in\omega} A_n$.  For each $%
n\in\omega$,  let $X_{n}=A_{n}\cup \{\infty \}$. If $A_n$ is finite, let $\mathbf{X}_n$ be the discrete space with the underlying set $X_n$. If $A_n$ is infinite, let $\mathbf{X}_n$  be the Alexandroff
compactification of the discrete space $\mathbf{A}_{n}$. Clearly, $\{\mathbf{%
X}_{n}: n\in \omega\}$ is a family of compact, metrizable, second-countable
spaces. Assuming that the product $\mathbf{X}=\prod_{n\in 
\omega}\mathbf{X}_{n}$ is metrizable and Loeb, we fix a Loeb function $f$ of $\mathbf{X}$ and a metric $d$ on $X$
which induces the product topology of $\mathbf{X}$. Obviously, using $d$, we can easily define a sequence $(d_n)_{n\in\omega}$ such that, for each $n\in\omega$, $d_n$ is a metric on $X_n$ which induces the topology of $\mathbf{X}_n$. For each $n\in\omega$, let $i_n=\min\{i\in\mathbb{N}: B_{d_n}(\infty, \frac{1}{i})\neq X_n\}$ and let $C(n)=\pi_n^{-1}(A_n\setminus B_{d_n}(\infty, \frac{1}{i_n}))$. The sets $C(n)$ are non-empty and closed in $\mathbf{X}$. We define a choice function $g$ for $\mathcal{A}$ by putting $g(A_n)=f(C(n))(n)$ for each $n\in\omega$. \medskip 

(c) Now, we fix $\mathcal{A}=\{A_n: n\in\omega\}$ and $\infty$ as in the proof to (b). We assume that each $A_n$ is a countably infinite set. We put $Y_n=A_n\cup\{\infty\}$ and $\mathbf{Y}_n=\langle Y_n, \mathcal{P}(Y_n)\rangle$. The discrete spaces $\mathbf{Y}_n$ are all Cantor completely metrizable and second-countable. Suppose that $\mathbf{Y}=\prod_{n\in\omega}\mathbf{Y}_n$ is Loeb. Let $f$ be a Loeb function of $\mathbf{Y}$. For each $n\in\omega$, we define a choice function $h_n$ of $\mathcal{P}(A_n)\setminus\{\emptyset\}$ as follows: if $Z\in\mathcal{P}(A_n)\setminus\{\emptyset\}$, then $h_n(Z)=\pi_n(f(\pi_n^{-1}(Z)))$. In much the same way, as in Remark \ref{r2.5}, we can define a sequence $(\leq_n)_{n\in\omega}$ such that each $\leq_n$ is a well-ordering on $A_n$. This implies that $\mathcal{A}$ is well-orderable.\medskip

(d) Assume $\mathbf{CAC}$ and consider any sequence $(\mathbf{X}_n)_{n\in\omega}$ of Cantor completely metrizable second-countable spaces. By $\mathbf{CAC}$, there exist sequences $(d_n)_{n\in\omega}$ and $(\mathcal{B}_n)_{n\in\omega}$ such that, for each $n\in\omega$, $\mathcal{B}_n$ is a countable base of $\mathbf{X}_n$, while $d_n$ is a Cantor complete metric on $X_n$ which induces the topology of $\mathbf{X}_n$. It was shown in \cite{kyr2} that the metric $d$ given by (\ref{4}) is Cantor complete, so $\mathbf{X}=\prod_{n\in\omega}\mathbf{X}_n$ is Cantor completely metrizable. For each $n\in\mathbb{N}$, we define $\mathcal{D}_n=\{U\times\prod_{i\in\omega\setminus n}X_i: U\in\prod_{i\in n}\mathcal{B}_i\}$. Since the collections $\mathcal{D}_n$ are all countable, under $\mathbf{CAC}$, the collection $\mathcal{D}=\bigcup_{n\in\mathbb{N}}\mathcal{D}_n$ is countable. This implies that $\mathbf{X}$ is second-countable. By Theorem \ref{t2.2}, $\mathbf{X}$ is Loeb.
\end{proof}

\begin{remark}
\label{r2.7} Theorem \ref{t2.6} clearly shows that, in some models of $%
\mathbf{ZF}$, countable products of Cantor completely metrizable second-countable spaces need not be Loeb. For instance, in Feferman-Levy
model $\mathcal{M}$9 of \cite{hr}, $\mathbb{R}$ is not well-orderable,
although $\mathbb{R}$ is a countable union of countable sets, so, by Theorem %
\ref{t2.6}(c), there exists in $\mathcal{M}$9 a countable product of Cantor
completely metrizable second-countable spaces which fails to be Loeb. Of
course, in view of Theorem \ref{t2.2}, all finite products of Cantor
completely metrizable second-countable spaces are Loeb in every model of $%
\mathbf{ZF}$.
\end{remark}

Mimicking the proof to Theorem 3.14 of \cite{kw} which asserts that $\mathbf{%
CAC}(\mathbb{R})$ and $\mathbf{CAC}_D(\mathbb{R})$ are equivalent, we can
deduce that the following theorem also holds true in $\mathbf{ZF}$:

\begin{theorem}
\label{t2.8} $\mathbf{CAC}_{\omega}(\mathbb{R})$ and $\mathbf{CAC}_{\omega
D}(\mathbb{R})$ are equivalent.
\end{theorem}

\begin{remark}
\label{r2.9} It is known that the statement \textquotedblleft $\mathbb{R}$
is sequential\textquotedblright\ implies $\mathbf{IDI}(\mathbb{R})$. This,
together with the fact that $\mathbb{R}$ is Loeb in $\mathbf{ZF}$, implies
that if every Loeb metrizable space were sequential, then $\mathbf{IDI}(%
\mathbb{R})$ would hold. Since $\mathbf{IDI}(\mathbb{R})$ fails in Cohen's
Basic Model $\mathcal{M}$1 of \cite{hr}, we see that there exists a
non-sequential Loeb metrizable space in $\mathcal{M}$1.
\end{remark}

The question which pops up at this point is whether it is provable in $\mathbf{ZF}$ that every sequential
metrizable space is Loeb. The following theorem gives a
negative answer to this question:

\begin{theorem}
\label{t2.10} (a) The statement: \textquotedblleft every sequential subspace
of $\mathbb{R} $ is Loeb\textquotedblright\ implies $\mathbf{CAC}_{\omega }(%
\mathbb{R})$.\newline
(b) The statement: \textquotedblleft every sequential metrizable space is
Loeb\textquotedblright\ is equivalent with $\mathbf{AC}$.
\end{theorem}

\begin{proof}
(a) Assume the contrary and fix a pairwise disjoint family $\mathcal{A}%
=\{A_{n}:n\in \mathbb{N}\}$ of countably infinite subsets of $\mathbb{R}$ without a
partial choice function. Without loss of generality we may assume that, for every $%
n\in \mathbb{N}$, $A_{n}\subseteq (\frac{1}{2n}, \frac{1}{2n-1})$. Let $X=\bigcup \mathcal{A%
}$ carry the Euclidean metric. To prove that $\mathbf{X}$ is sequential, let us consider a
non-empty sequentially closed in $\mathbf{X}$ set $G$ and a point $x\in \text{cl}_{\mathbf{X}}(G)$. Then $%
x\in A_{n_x}$ for some $n_x\in \mathbb{N}$. Since $x\in \text{cl}_{\mathbf{X}}(G\cap A_{n_x})$ and 
$A_{n_x}$ is countable, it follows that some sequence of points of $G\cap A_{n_x}
$ converges in $\mathbf{X}$ to $x$. Hence, $x\in G$ because $G$ is sequentially closed in $\mathbf{X}$. Therefore, $\text{cl}_{\mathbf{X}}(G)=G$ and $\mathbf{X}$ is sequential. However, each member of $\mathcal{A}$
is closed in $\mathbf{X}$ but $\mathcal{A}$ has no choice function, so $\mathbf{X}$ is not Loeb.\medskip 

(b) Since $\mathbf{AC}$ implies that every topological space is Loeb, it suffices to show that, in every model of $\mathbf{ZF}$ in which $\mathbf{AC}$ fails, there exists a sequential metrizable space which is not Loeb. To see this, let us suppose that $J$ is a non-empty set and  $\mathcal{C}=\{C_j: j\in J\}$ is a pairwise disjoint family of
non-empty sets. Now, let $\mathbf{Y}=\langle Y, \mathcal{P}(Y)\rangle$ where  $Y=\bigcup \mathcal{C}$%
.  Clearly, the discrete space $\mathbf{Y}$ is a metrizable sequential space which is not Loeb.
\end{proof}

\section{Sequential spaces and their finite products}

We recommend \cite{F1}, \cite{F2} and Section 1.6 of \cite{En} as very good
sources of fundamental facts about sequential spaces in $\mathbf{ZFC}$. It
is known that some theorems on sequential and Fr\'echet-Urysohn spaces, given
in \cite{F1}, \cite{F2} and \cite{En}, fail in some models of $\mathbf{ZF}$
(see, for instance, \cite{gg2} and \cite{herl}). Let us recall that Theorem
4.54 of \cite{herl} shows that $\mathbb{R}$ is Fr\'echet-Urysohn in a model $%
\mathcal{M}$ of $\mathbf{ZF}$ if and only if $\mathbf{CAC}(\mathbb{R})$
holds in $\mathcal{M}$, while Theorem 4.55 of \cite{herl} shows that $%
\mathbb{R}$ is sequential in a model $\mathcal{M}^{\ast}$ of $\mathbf{ZF}$
if and only if the collection of all non-empty complete subspaces of $%
\mathbb{R}$ has a choice function in $\mathcal{M}^{\ast}$. It is known that $%
\mathbb{R}$ is not sequential, for instance, in Cohen's model $\mathcal{M}$1
of \cite{hr}. Fr\'echet-Urysohn spaces are called shortly Fr\'echet spaces,
for instance, in \cite{En}, \cite{F1}, \cite{F2} and \cite{herl}. Problems
6.6 and 6.11 of \cite{OPWZ} that were unsolved in \cite{OPWZ} and have been
solved in \cite{kw} recently, are both relevant to the question whether, for
every positive integer $m$, $\mathbb{R}^m$ is sequential in every model of $%
\mathbf{ZF}$ in which $\mathbb{R}$ is sequential. In this section, we give a
new, quite different than in \cite{kw}, $\mathbf{ZF}$-proof to the theorem
that if $\mathbb{R}$ is sequential, then $\mathbb{R}^m$ is sequential for
every $m\in\mathbb{N}$. Moreover, we offer an overview of some known $%
\mathbf{ZFC}$-results on sequential spaces, especially the ones concerning
Cartesian products, and we give new $\mathbf{ZF}$-proofs to some of them. In
particular, we show that it is provable in $\mathbf{ZF}$ that if $\mathbf{X}$
is a sequential locally countably compact space, then $\mathbf{X}\times%
\mathbf{Y}$ is sequential for every sequential space $\mathbf{Y}$. We also
obtain new $\mathbf{ZF}$-theorems on sequential spaces. Part of our results
involves Loeb spaces. For instance, we prove that if $\mathbf{X}$ is a
locally sequentially compact sequential Loeb space, then $\mathbf{X}\times%
\mathbf{Y}$ is sequential for every sequential space $\mathbf{Y}$. We also
prove that a first-countable, countably compact regular space $\mathbf{X}$
is sequential if and only if $\mathbf{PCAC}(\mathbf{X}, scl)$ holds.

We recall that, for topological spaces $\mathbf{X}$ and $\mathbf{Y}$, a mapping $%
f:\mathbf{X}\to\mathbf{Y}$ is said to be \textit{quotient} if $f$ is a
surjection such that, for every subset $W$ of $\mathbf{Y}$, the set $%
f^{-1}(W)$ is open in $\mathbf{X}$ if and only if $W$ is open in $\mathbf{Y}$%
.

We need the following results of \cite{F1} which are easily provable in $%
\mathbf{ZF}$:

\begin{proposition}
\label{p3.1} (Cf. Proposition 1.2 of \cite{F1}.) If $\mathbf{X}$ is a
sequential space and $\mathbf{Y}$ is a topological space such that there
exists a quotient mapping of $\mathbf{X}$ onto $\mathbf{Y}$, then $\mathbf{Y}
$ is sequential.
\end{proposition}

\begin{corollary}
\label{c3.2} (Cf. Corollary 1.4 of \cite{F1}.) If $f:\mathbf{X}\to\mathbf{Y} 
$ is a continuous surjection of a sequential space $\mathbf{X}$ to a
topological space $\mathbf{Y}$ and $f$ is an open or closed mapping, then
the space $\mathbf{Y}$ is sequential.
\end{corollary}

\begin{proposition}
\label{p3.3} (Cf. Proposition 1.9 of \cite{F1}) All closed and all open
subspaces of sequential spaces are sequential.
\end{proposition}

\begin{proposition}
\label{p3.4} (Cf. Corollary 1.5 of \cite{F1}) If $\mathbf{X}_1$ and $\mathbf{X}_2 
$ are non-empty spaces such that their product $\mathbf{X}_1\times\mathbf{X}%
_2$ is sequential, then both $\mathbf{X}_1$ and $\mathbf{X}_2$ are
sequential.
\end{proposition}

The following proposition shows that Corollary 1.5 of \cite{F1} is
unprovable in $\mathbf{ZF}$:

\begin{proposition}
\label{r3.5} $\mathbf{AC}$ is equivalent with the following statement:

(${\star}$) For every non-empty set $J$ and every collection $\{\mathbf{X}%
_j: j\in J\}$ of non-empty spaces, it holds true that if $\prod_{j\in J}%
\mathbf{X}_j$ is sequential, then $\mathbf{X}_j$ is sequential for each $%
j\in J$.
\end{proposition}

\begin{proof} If $\mathbf{AC}$ holds and $\{%
\mathbf{X}_j: j\in J\}$ is a collection of non-empty spaces such that $%
\mathbf{X}=\prod_{j\in J}\mathbf{X}_j$ is sequential, then, for each $j_0\in J$, the projection $\pi_{j_0}:\mathbf{X}\to\mathbf{X}_{j_0}$ is a quotient mapping, so $\mathbf{X}_{j_0}$ is sequential by Proposition \ref{p3.1}. On the other hand, assuming that $\mathcal{M}$ is a model of $\mathbf{ZF}$ in which $\mathbf{AC}$ is false, we fix in $\mathcal{M}$  a non-empty collection $\mathcal{A}$ of non-empty sets which does not have a choice function in $\mathcal{M}$. For $A\in\mathcal{A}$, let $\mathbf{A}$ be the discrete in $\mathcal{M}$ space with its underlying set $A$. Let $\mathbf{Y}$ be any non-sequential topological space in $\mathcal{M}$ and let $\mathcal{X}=\{\mathbf{A}: A\in\mathcal{A}\}\cup\{\mathbf{Y}\}$. Then the Tychonoff product $\prod_{\mathbf{X}\in\mathcal{X}}\mathbf{X}$ in $\mathcal{M}$ is a sequential space, while $\mathbf{Y}$ is not sequential. 
\end{proof}

The following theorem is a modification of Proposition 1.12 of \cite{F1}
(see also Exercise 2.4.G(b) in \cite{En}):

\begin{theorem}
\label{t3.6} (Cf. Proposition 1.12 of \cite{F1}. For a non-empty topological
space $\mathbf{X}=\langle X, \tau\rangle$, let $\mathcal{S}$ be the set of
all elements $s\in X^{\omega}$ such that the sequence $(s(n))_{n\in\mathbb{N}%
}$ converges in $\mathbf{X}$ to $s(0)$. For each $s\in\mathcal{S}$, let $%
Z_s=\{s\}\times(\{0\}\cup\{\frac{1}{n}: n\in\mathbb{N}\})$ and let $f_s:
Z_s\to X$ be defined as follows; $f_s(\langle s, 0\rangle)=s(0) $ and $%
f_s(\langle s, \frac{1}{n}\rangle)=s(n)$ for each $n\in\mathbb{N}$. If $s\in%
\mathcal{S}$, let $\mathbf{Z}_s$ denote the topological space $\langle Z_s,
\tau_s\rangle$ where $\tau_s=\{ \{s\}\times (U\cap(\{0\}\cup\{\frac{1}{n}:
n\in\mathbb{N}\})): U\in\tau_{nat}\}$. Let $\mathbf{Z}(\mathbf{X})$ be the
topological sum $\oplus_{s\in\mathcal{S}}\mathbf{Z}_s$ of the spaces $%
\mathbf{Z}_s$. Let $f:\mathbf{Z}(\mathbf{X})\to\mathbf{X}$ be defined as
follows: $f(z)=f_s(z)$ whenever $s\in\mathcal{S}$ and $z\in Z_s$. Then $f$
is a quotient mapping if and only if the space $\mathbf{X}$ is sequential.
\end{theorem}

\begin{definition}
\label{d3.7} (Cf. Exercise 3.10.J of \cite{En}.) A topological space $%
\mathbf{X}$ is called \textit{locally sequentially compact} (respectively, 
\textit{locally countably compact}) if $\mathbf{X}$ is a $T_0$-space and, for
every open in $\mathbf{X}$ set $V$ and each $x\in V$, there exists an open
neighbourhood $U$ of $x$ such that $\text{cl}_{\mathbf{X}}(U)$ is
sequentially compact (respectively, countably compact) in $\mathbf{X}$ and $%
\text{cl}_{\mathbf{X}}(U)\subseteq V$.
\end{definition}

Let us notice that every regular $T_0$-space is Hausdorff. This is why, in our article, all locally sequentially compact  and all locally countably compact spaces are Hausdorff.  We do not consider modifications of Definition \ref{d3.7} to more general cases. We recommend \cite{kyr} and \cite{mm} to learn more about sequential
compactness and relevant concepts in the absence of $\mathbf{CAC}$.

Now, we are in a position to show several new results on sequential spaces
and their finite products in $\mathbf{ZF}$. Every ordinal number of von
Neumann, as a topological space, is considered with the topology induced by
its standard well-ordering.

\begin{lemma}
\label{l3.8} Let $\mathbf{X}_1, \mathbf{X}_2$ and  $\mathbf{Y}$ be sequential
spaces such that $\mathbf{X}_{1}$ is a locally sequentially compact space
which satisfies $\mathbf{PCAC}(\mathbf{X}_1, scl)$. Then the following
conditions are satisfied:

\begin{enumerate}
\item[(i)] for every sequentially open in $\mathbf{X}_1\times\mathbf{X}_2$
set $O$ and for each $t\in X_2$, the set $O_t=\{x\in X_1: \langle x,
t\rangle\in O\}$ is open in $\mathbf{X}_1$;

\item[(ii)] for every sequentially compact in $\mathbf{X}_1$ set $C$, for
every sequentially open in $\mathbf{X}_1\times\mathbf{X}_2$ set $O$, the set 
$\{t\in X_2: C\times\{t\}\subseteq O\}$ is open in $\mathbf{X}_2$.
\end{enumerate}
\end{lemma}

\begin{proof}  Assume that $O$ is a sequentially open subset of $\mathbf{X}_1\times\mathbf{X}_2$.

 (i) Let $t\in X_2$ and let
$(x_n)_{n\in\omega}$ be a sequence of points of $X_1\setminus O_t$ which converges in $\mathbf{X}_1$ to a point $x$. Then $(\langle x_n, t\rangle)_{n\in\omega}$ is a sequence of points of $(X_1\times\{t\})\setminus O$ which converges in $\mathbf{X}_1\times\mathbf{X}_2$ to the point $\langle x, t\rangle$. Since the set  $(X_1\times\{t\})\setminus O$ is sequentially closed in the sequential space $\mathbf{X}_1\times\{t\}$, we infer that $\langle x, t\rangle\notin O$. Hence $x\notin O_t$. This proves that $X_1\setminus O_t$ is sequentially closed in $\mathbf{X}_1$. In consequence, $O_t$ is sequentially open in $\mathbf{X}_1$. Since $\mathbf{X}_1$ is sequential, $O_t$ is open in $\mathbf{X}_1$.

(ii) Now, assume that $C$ is sequentially compact in $\mathbf{X}_1$. Since the space $\mathbf{X}_1$ is sequential and Hausdorff, the set $C$ is closed in $\mathbf{X}_1$. Let $V=\{t\in X_2: C\times\{t\}\subseteq O\}$. Suppose that $V$ is not open in $\mathbf{X}_2$. Then, since $\mathbf{X}_2$ is sequential, the set $X_2\setminus V$ is not sequentially closed in $\mathbf{X}_2$. Thus, there exists a sequence $(y_n)_{n\in\omega}$ of points of $X_2\setminus V$ which converges in $\mathbf{X}_2$ to a point $y\in V$. For each $n\in\omega$, the set $O_n=\{x\in X_1: \langle x, y_n\rangle\in O\}$ is open in $\mathbf{X}_1$. Hence, the sets $A_n=C\setminus O_n$ are closed in $\mathbf{X}_1$. For each $n\in\omega$, the set $A_n$ is non-empty because $y_n\in X_2\setminus V$. Since $\mathbf{PCAC}(\mathbf{X}_1, scl)$ holds, there exist an infinite set $N\subseteq\omega$ and a point $z\in\prod_{n\in N}A_n$. By the sequential compactness of $C$, there exists a subsequence $(z_{n_k})_{k\in\omega}$ of $(z(n))_{n\in N}$ such that $(z_{n_k})_{k\in\omega}$ converges in $\mathbf{X}_1$ to some point $z\in C$. Then $(\langle z_{n_k}, y_{n_k}\rangle)_{k\in\omega}$ converges in $\mathbf{X}_1\times\mathbf{X}_2$ to $\langle z, y\rangle$. Since $y\in V$, while $z\in C$, we deduce that $\langle z, y\rangle\in O$. This, together with the assumption that $O$ is sequentially open,  implies that there exists $k_0\in\omega$ such that $\langle z_{n_k}, y_{n_k}\rangle\in O$ for each $k\in\omega\setminus k_0$. This is impossible because $z_{n_k}\notin O_{n_k}$ for each $k\in\omega$. The contradiction obtained proves (ii).
\end{proof}

A known result of $\mathbf{ZFC}$ is that the product $\mathbf{X}_1\times%
\mathbf{X}_2$ of sequential spaces $\mathbf{X}_1$ and $\mathbf{X}_2$ is
sequential if $\mathbf{X}_1$ is a locally sequentially compact space (see,
for instance, Exercise 3.10.J in \cite{En} and Theorem 4.2 in \cite{mich}).
Another known result of $\mathbf{ZFC}$ is that if $\mathbf{X}_1$ and $\mathbf{X}_2
$ are sequential spaces such that $\mathbf{X}_1$ is locally sequentially
compact, while $g:\mathbf{X}_2\to\mathbf{Y}$ is a quotient mapping, then $%
\text{id}_{X_1}\times g$ is quotient (see Theorem 4.2 of \cite{mich} and
Exercise 3.10.J(b) of \cite{En}). Our modification for $\mathbf{ZF}$ of the
above-mentioned results is the following theorem:

\begin{theorem}
\label{t3.9} Let $\mathbf{X}_1, \mathbf{X}_2$ and $\mathbf{Y}$ be sequential
spaces such that $\mathbf{X}_{1}$ is a locally sequentially compact space
for which $\mathbf{PCAC}(\mathbf{X}_1, scl)$ holds. Suppose that $g:\mathbf{X%
}_2\to\mathbf{Y}$ is a quotient mapping. Then:

\begin{enumerate}
\item[(i)] $\mathbf{X}_1\times\mathbf{X}_2$ is sequential;

\item[(ii)] the mapping $f=\text{id}_{X_1}\times g:\mathbf{X}_1\times\mathbf{%
X}_2\to\mathbf{X}_1\times\mathbf{Y}$ is quotient.
\end{enumerate}
\end{theorem}

\begin{proof}
(i) Fix a sequentially open subset $O$ of $\mathbf{X}_{1}\times \mathbf{X}_{2}$. We show that $O$ is open.  In view of Lemma \ref{l3.8}(i), for every $t\in X_{2}$, the set
$O_{t}=\{x\in X_{1}:\langle x, t\rangle\in O\}$
is open in $\mathbf{X}_1$.  To see that $O$ is open, fix $\langle s, t\rangle\in O$. Clearly, $s\in O_{t}$%
. Since $\mathbf{X}_1$ is a locally sequentially compact space, there exists a sequentially compact neighbourhood $C$ of $s$ with $C\subseteq O_{t}$.  Clearly, $C\times \{t\}\subseteq O$. Let 
$V=\{y\in X_{2}:C\times \{y\}\subseteq O\}$. It follows from Lemma \ref{l3.8}(ii) that $V$ is open in $\mathbf{X}_2$. Since $\langle s, t\rangle\in C\times V\subseteq O$, we infer that $O$ is open in $\mathbf{X}_1\times\mathbf{X}_2$. This implies that $\mathbf{X}_1\times\mathbf{X}_2$ is sequential.

(ii)  Let $W$ be a subset of $X_{1}\times Y$ such that $f^{-1}(W)$ is open in $%
\mathbf{X}_{1}\times \mathbf{X}_{2}$. To prove that $f$ is quotient, it suffices to check that $W$ is open
in $\mathbf{X}_{1}\times \mathbf{Y}$. Let $\langle x_{0},z_{0}\rangle \in W$%
. Take $y_{0}\in g^{-1}(z_{0})$ and a neighbourhood $U$ of $x_{0}$ such that 
$D=\text{cl}_{\mathbf{X}_1}(U)$ is sequentially compact and $D\times
\{y_{0}\}\subseteq f^{-1}(W)$. Then $D\times
g^{-1}(z_{0})\subseteq f^{-1}(W)$. Let $H=\{y\in Y: D\times
g^{-1}(y)\subseteq f^{-1}(W)\}$. We have $\langle x_{0},z_{0}\rangle \in
U\times H\subseteq W$. To complete the proof, it suffices to show that $H$
is open in $\mathbf{Y}$. Since $g$ is quotient, it suffices to check that $%
g^{-1}(H)$ is open in $\mathbf{X}_{2}$. We notice that $g^{-1}(H)=\{z\in
X_{2}:D\times \{z\}\subseteq f^{-1}(W)\}$. It follows from Lemma \ref{l3.8}(ii) that  $g^{-1}(H)$
is open in $\mathbf{X}_2$. Hence $H$ is open in $\mathbf{Y}$.
\end{proof}

\begin{corollary}
\label{c3.10} For every sequential space $\mathbf{X}$, the product $\mathbf{X%
}\times (\omega+1)$ is sequential.
\end{corollary}

\begin{proof} In the light of Theorem \ref{t3.9}, it suffices to notice that $\omega+1$ is a sequential, locally sequentially compact space for which $\mathbf{PCAC}(\omega+1, scl)$ holds true. In fact, $\omega+1$ is Loeb.
\end{proof}

Let us notice that the following Whitehead Theorem (see Lemma 4 of \cite{Wh}%
, Theorem 3.3.17 in \cite{En} and Theorem 2.1 of \cite{mich}) is provable in 
$\mathbf{ZF}$:

\begin{theorem}
\label{t3.11} (Whitehead) Let $\mathbf{X}, \mathbf{Y}$ and $\mathbf{Z}$ be
topological spaces such that $\mathbf{X}$ is a locally compact Hausdorff
space. Suppose that $g:\mathbf{Y}\to\mathbf{Z}$ is a quotient mapping. Then
the mapping $f=\text{id}_{X}\times g: \mathbf{X}\times\mathbf{Y}\to\mathbf{X}%
\times\mathbf{Z}$ is quotient.
\end{theorem}

Condition (ii) of Theorem \ref{t3.9} is a $\mathbf{ZF}$-version of the
theorem given in Exercise 3.10.J(b) of \cite{En} (see also Theorem 4.2 of 
\cite{mich}). Our proof to Theorem \ref{t3.9}(i) gives a distinct solution
of Exercise 3.10.J(c) of \cite{En} than the solution suggested in \cite{En}.

Theorem \ref{t3.12} below is known (see, for instance, \cite{Bo}, \cite{mich}
and Exercise 3.3.J of \cite{En}); however, the known proofs to it are not
proofs in $\mathbf{ZF}$ because $\mathbf{AC}$ is involved in them. This is
why we offer two distinct new $\mathbf{ZF}$-proofs to it. To some extent,
our first proof to Theorem \ref{t3.12} is similar to that shown in the hint
to Exercise 3.3.J of \cite{En} but, since we need to avoid $\mathbf{AC}$,
there are also essential differences between our proof and the known $%
\mathbf{ZFC}$-proofs.

\begin{theorem}
\label{t3.12} Let $\mathbf{X}$ be a locally compact Hausdorff sequential
space. Then $\mathbf{X}\times\mathbf{Y}$ is sequential for every sequential
space $\mathbf{Y}$.
\end{theorem}

\begin{proof} First, we notice that, by Corollary \ref{c3.10}, the space $\mathbf{X}\times (\omega+1)$ is sequential. Let $\mathbf{Y}$ be a sequential space and let $\mathbf{Z}(\mathbf{Y})$ be the topological sum of copies of $\omega+1$ described in Theorem \ref{t3.6}. Let $g:\mathbf{Z}(\mathbf{Y})\to\mathbf{Y}$ be a quotient mapping. Since $\mathbf{X}\times (\omega+1)$ is sequential,  we infer that $\mathbf{X}\times\mathbf{Z}(\mathbf{Y})$ is a sequential space. By Theorem \ref{t3.11}, the space $\mathbf{X}\times\mathbf{Y}$ is a quotient image of the sequential space $\mathbf{X}\times\mathbf{Z}(\mathbf{Y})$. In view of Proposition \ref{p3.1}, the space $\mathbf{X}\times\mathbf{Y}$ is sequential. 
\end{proof}

Using mathematical induction, one can easily deduce from Theorem \ref{t3.12}
that the following corollary is provable in $\mathbf{ZF}$:

\begin{corollary}
\label{c3.13} For $n\in \mathbb{N}$, let $\{\mathbf{X}_{i}:i\in n\}$ be a
collection of locally compact Hausdorff sequential spaces. Then the product 
$\mathbf{X}=\prod_{i\in n}\mathbf{X}_{i}$ is sequential.
\end{corollary}

It is worth noticing that the following important Theorem 3.5 of \cite{kw}
follows directly from Corollary \ref{c3.13} and Proposition \ref{p3.4}:

\begin{corollary}
\label{c3.14} (Theorem 3.5 of \cite{kw}) For every $n\in\mathbb{N}$, the
space $\mathbb{R}^n$ is sequential if and only if $\mathbb{R}$ is sequential.
\end{corollary}

Part (a) of the following new theorem is to be applied to a $\mathbf{ZF}$%
-proof that, in Theorem \ref{t3.12}, local compactness can be replaced by
local countable compactness:

\begin{theorem}
\label{t3.15} (a) If a countably compact space $\mathbf{X}$ is sequential,
then it satisfies $\mathbf{PCAC}(\mathbf{X}, scl)$.\newline
(b) If $\mathbf{X}$ is a first-countable regular space such that $\mathbf{%
PCAC}(\mathbf{X}, scl)$ holds, then $\mathbf{X}$ is sequential.\newline
(c) A first-countable, regular countably compact space $\mathbf{X}$ is
sequential if and only if $\mathbf{PCAC}(\mathbf{X}, scl)$ holds.
\end{theorem}

\begin{proof}
(a) Suppose that $\mathbf{X}$ is a countably compact sequential space such that $\mathbf{PCAC}(\mathbf{X}, scl)$ is not satisfied. Let $\mathcal{E}%
=\{E_{n}:n\in \omega\}$ be a family of closed subsets of $\mathbf{X}$
without a partial choice function. For every $n\in \omega$, let 
\begin{equation*}
F_{n}=\bigcup \{E_{i}:i\in\omega\setminus n\}\text{.}
\end{equation*}
We claim that $F_{n}$ is closed in $\mathbf{X}$. To this aim, since $\mathbf{X}$ is sequential, it suffices to show that $F_n$ is sequentially closed in $\mathbf{X}$. We fix a sequence $(x_k)_{k\in\omega}$ of points of $F_n$ which converges in $\mathbf{X}$ to a point $x$. From the assumption that $\mathcal{E}$ does not have a partial choice function, we deduce that there exists $k_0\in\omega\setminus n$ such that $x_k\in\bigcup\{E_i: i\in k_0\setminus n\}$ for each $k\in\omega$. Since all members of $\mathcal{E}$ are closed in $\mathbf{X}$, we infer that $x\in\bigcup\{E_i: i\in k_0\setminus n\}$. Hence $x\in F_n$, so $F_n$ is sequentially closed in $\mathbf{X}$. It follows from the countable compactness of $\mathbf{X}$ that  $\bigcap \{F_{n}:n\in 
\omega\}\neq \emptyset $. It is straightforward to see that any point in
the latter intersection belongs to infinitely many members of $\mathcal{E}$.
Hence, $\mathcal{E}$ has a partial choice function, contradicting our
hypothesis.\smallskip\ 

(b) Assume that $\mathbf{PCAC}(\mathbf{X}, scl)$ holds and $\mathbf{X}$ is a first-countable regular space. Let $A$ be a sequentially closed set in $\mathbf{X}$ and let $y\in\text{cl}_{\mathbf{X}}(A)$. Let $\{U_n: n\in\omega\}$ be a countable base of neighbourhoods of $y$ such that $U_{n+1}\subseteq U_n$ for each $n\in\omega$. Let $G_n=A\cap\text{cl}_{\mathbf{X}}(U_n)$ for each $n\in\omega$. The sets $G_n$ are sequentially closed in $\mathbf{X}$, so the collection $\{G_n: n\in\omega\}$ has a partial choice function. Let $N$ be an infinite subset of $\omega$ such that the product $\prod_{n\in N}G_n$ is non-empty. Since $\mathbf{X}$ is regular, if $f\in\prod_{n\in N}G_n$, then the sequence $(f(n))_{n\in N}$ is convergent in $\mathbf{X}$ to the point $y$. Hence $y\in A$, so $A$ is closed in $\mathbf{X}$.  This completes the proof of (b). Condition (c) is an immediate consequence of (a) and (b).
\end{proof}

Now, we are able to show that the following known theorem of $\mathbf{ZFC}$
(see Theorem 4.2 of \cite{mich}) is provable in $\mathbf{ZF}$:

\begin{theorem}
\label{t3.16} Let $\mathbf{X}_1, \mathbf{X}_2$ and $\mathbf{Y}$ be sequential
spaces such that $\mathbf{X}_{1}$ is locally countably compact. Suppose that 
$g:\mathbf{X}_2\to\mathbf{Y}$ is a quotient mapping. Then:

\begin{enumerate}
\item[(i)] $\mathbf{X}_1\times\mathbf{X}_2$ is sequential;

\item[(ii)] the mapping $f=\text{id}_{X_1}\times g:\mathbf{X}_1\times\mathbf{%
X}_2\to\mathbf{X}_1\times\mathbf{Y}$ is quotient.
\end{enumerate}
\end{theorem}

\begin{proof} Let $C$ be a countably compact closed subset of $\mathbf{X}_1$. Since $\mathbf{X}_1$ is sequential, the subspace $\mathbf{C}$ of $\mathbf{X}_1$ is sequential. By Theorem \ref{t3.15}(a), $\mathbf{PCAC}(\mathbf{C}, scl)$ holds. Since the space $\mathbf{C}$ is both sequential and countably compact, $\mathbf{C}$ is sequentially compact (see the proof to Theorem 3.10.31 of \cite{En}).  In the light of the proof to Lemma \ref{l3.8}(ii), for every sequentially open in $\mathbf{X}_1\times\mathbf{X}_2$ set $O$, the set 
$\{t\in X_2: C\times\{t\}\subseteq O\}$ is open in $\mathbf{X}_2$. Now, in much the same way, as in the proof to Theorem \ref{t3.9},  we can show  that $\mathbf{X}_1\times\mathbf{X}_2$ is sequential, and (ii) is also satisfied. 
\end{proof}

\begin{remark}
\label{r3.17}  Our second $\mathbf{ZF}$-proof to Theorem \ref{t3.12} is to
notice that, since every locally compact Hausdorff space is locally
countably compact, Theorem \ref{t3.12} is an immediate consequence of
Theorem \ref{t3.16}.
\end{remark}

In the light of Theorem 8 of \cite{hkt}, if $\mathbf{X}_1$ and $\mathbf{X}_2$ are
Loeb metric spaces such that the canonical projection $\pi_1:\mathbf{X}%
_1\times\mathbf{X}_2\to\mathbf{X}_1$ is closed, then $\mathbf{X}_1\times%
\mathbf{X}_2$ is Loeb. We need the following more general lemma:

\begin{lemma}
\label{l3.18} Let $\mathbf{X}_1$ and $\mathbf{X}_2$ be Loeb spaces such that
the projection $\pi_1:\mathbf{X}_1\times\mathbf{X}_2\to\mathbf{X}_1$ is
closed. Then $\mathbf{X}_1\times\mathbf{X}_2$ is Loeb.
\end{lemma}

\begin{proof} Let $f_i$ be a Loeb function of $\mathbf{X}_i$ for $i\in\{1,2\}$.  For a non-empty closed set $A$ of $\mathbf{X}_1\times\mathbf{X}_2$, let $x_1(A)=f_{1}(\pi_1(A))$. The set $B_A=\{x\in X_2:\langle x_1(A), x\rangle\in A\}$ is non-empty and closed in $\mathbf{X}_2$. We define $x_2(A)=f_2(B_A)$. Then, by putting $g(A)=\langle x_1(A), x_2(A)\rangle$ for each non-empty closed set $A$ of $\mathbf{X}_1\times\mathbf{X}_2$, we define a Loeb function $g$ of $\mathbf{X}_1\times\mathbf{X}_2$.
\end{proof}

\begin{theorem}
\label{t3.19} Suppose that $\mathbf{X}_1$ and $\mathbf{X}_2$ are sequential Loeb
spaces such that $\mathbf{X}_1$ is a $T_1$-space and $\mathbf{X}_2$ is a
sequentially compact $T_3$-space. Then the product $\mathbf{X}_1\times 
\mathbf{X}_2$ is a sequential Loeb space.
\end{theorem}

\begin{proof}
That $\mathbf{X}_{1}\times \mathbf{X}_{2}$ is sequential follows at once from
Theorem \ref{t3.9}. In view of Lemma \ref{l3.18}, to prove that $\mathbf{X}_{1}\times \mathbf{X}_{2}$ is
Loeb, it suffices,  to show that the projection $%
\pi _{2}:\mathbf{X}_1\times\mathbf{X}_2\to\mathbf{X}_2$ is a closed map. Let $f_{i}$ be a Loeb function of $\mathbf{X%
}_{i}$ for $i\in\{1, 2\}$. Fix a closed subset $G$ of $\mathbf{X}_{1}\times \mathbf{X}_{2}$. We
claim that $G_{1}=\pi _{1}(G)$ is sequentially closed. To this end, fix a convergent in $\mathbf{X}_1$ to a point $x$
sequence $(x_{n})_{n\in \omega}$ of points of $G_{1}$. We show that $x\in G_{1}$. For every $n\in \omega$, the subspace 
\[
l_{n}=\{x_{n}\}\times \mathbf{X}_{2},
\]%
of $\mathbf{X}_1\times\mathbf{X}_2$ is homeomorphic  with $\mathbf{X}_{2}$, so $l_n$ is sequential. Since $%
L_{n}=l_{n}\cap G$ is a sequentially closed subset of $l_{n}$ and $l_n$ is closed in $\mathbf{X}_1\times\mathbf{X}_2$, it follows
that $L_{n}$ is closed in $\mathbf{X}_1\times\mathbf{X}_2$. For every $n\in\omega$, the set $K_{n}=\pi _{2}(L_{n})$ is non-empty and closed in $\mathbf{X}_{2}$, so we can define  $y_{n}=f_2(K_n)$. It follows from the sequential compactness of $\mathbf{X}%
_{2}$ that $(y_{n})_{n\in \mathbb{N}}$
has a subsequence $(y_{k_{n}})_{n\in \mathbb{N}}$ converging in $\mathbf{X}_2$ to some point $%
y$. Then $(\langle x_{k_{n}},y_{k_{n}}\rangle)_{n\in\omega}$ is a sequence of
points of $G$ converging to $\langle x,y\rangle$. Thus, $\langle x,y\rangle\in G$ and $x\in G_{1}$ as
required. Since $\mathbf{X}_{1}$ is sequential, it follows that $G_{1}$ is
closed in $\mathbf{X}_{1}$.\medskip 
\end{proof}

The following problems are unsolved:

\begin{problem}
\label{prob3.20} Is there a model of $\mathbf{ZF}$ in which there exist a
locally sequentially compact sequential space $\mathbf{X}$ and a sequential
space $\mathbf{Y}$ such that $\mathbf{X}\times\mathbf{Y}$ is not sequential?
\end{problem}

\begin{problem}
\label{prob3.21} Is there a model of $\mathbf{ZF}$ in which there are
metrizable sequential spaces $\mathbf{X}$ and  $\mathbf{Y}$ such that $\mathbf{X}%
\times\mathbf{Y}$ is not sequential?
\end{problem}

\section{On when $\mathbb{R}$ is sequential and related problems}

In this section, we pay a special attention to spaces which are sequential
when $\mathbb{R}$ is sequential. Among other results, we prove that $\mathbb{%
R}$ is sequential if and only if every second-countable compact Hausdorff
space is sequential. In the spirit of Theorem 4.55 of \cite{herl}, we give a
number of necessary and sufficient conditions for a Cantor complete
second-countable metric space to be sequential. In particular, we prove that
a Cantor complete second-countable metric space $\mathbf{Y}$ is sequential
if and only if $\mathbf{PCAC}(\mathbf{Y},scl)$ holds. We show that it is
relatively consistent with $\mathbf{ZF}$ that there exists a compact
metrizable space $\mathbf{Y}$ which satisfies $\mathbf{PCAC}(\mathbf{Y},scl)$
and does not satisfy $\mathbf{CAC}(\mathbf{Y},scl)$. We formulate several
other sentences relevant to sequentiality and prove that they are independent
of $\mathbf{ZF}$.

By a deep analysis of Problem 4.5.9 of \cite{En}, as well as of the proofs
to Theorems 2.3.26 and 3.2.2 given in \cite{En}, we deduce the following
lemma:

\begin{lemma}
\label{l4.1} It is provable in $\mathbf{ZF}$ that every non-empty
second-countable compact Hausdorff space is a continuous image of the Cantor
cube $\{0,1\}^{\omega}$.
\end{lemma}

\begin{theorem}
\label{t4.2} The following conditions are all equivalent:

\begin{enumerate}
\item[(i)] every second-countable compact Hausdorff space is sequential;

\item[(ii)] the Hilbert cube $[0,1]^{\omega}$ is sequential;

\item[(iii)] the interval $[0, 1]$ is sequential;

\item[(iv)] the Cantor cube $\{0,1\}^{\omega}$ is sequential;

\item[(v)] $\mathbb{R}$ is sequential.
\end{enumerate}
\end{theorem}

\begin{proof} The implications (i)$\rightarrow$(ii)$\rightarrow$(iii) are obvious.  Let $C$ denote the Cantor ternary set. It is well known that $\mathbf{C}$, as a closed bounded subspace of $\mathbb{R}$, is compact, while the function $H:\{0,1\}^{\omega }\rightarrow \mathbf{C}$ given by: 
\[
H(f)=\sum\limits_{n\in \omega }\frac{2f(n)}{3^{n+1}} 
\]%
is a homeomorphism. If $[0,1]$ or $\mathbb{R}$ is sequential, so is $\mathbf{C}$ by Proposition \ref{p3.3}. Hence (iv) follows from (iii) and from (v). That (iv) implies (i) follows from Lemma \ref{l4.1} and Proposition \ref{p3.1}. Finally, since $\mathbb{R}$ is homeomorphic with the open subspace $(0, 1)$ of $[0,1]$, we deduce from Proposition \ref{p3.3} that (iii) implies (v).
\end{proof}

The following corollary of Theorem \ref{t4.2} shows that we have just
obtained a new proof to our Corollary \ref{c3.14} and to Theorem 3.5 of \cite%
{kw} which, together with Theorem 3.6 of \cite{kw}, gives solutions to
Problems 6.6 and 6.11 of \cite{OPWZ}:

\begin{corollary}
\label{c4.3} For every $n\in\mathbb{N}$, the following are equivalent:

\begin{enumerate}
\item[(i)] $\mathbb{R}$ is sequential;

\item[(ii)] $[0,1]^{n}$ is sequential;

\item[(iii)] $(0,1)^{n}$ is sequential,

\item[(iv)] $\mathbb{R}^n$ is sequential.
\end{enumerate}
\end{corollary}

\begin{proof} Let $n\in\mathbb{N}$. If (i) holds, then, by Theorem \ref{t4.2}, the space $[0,1]^{\omega}$ is sequential, so $[0,1]^{n}$ is sequential as homeomorphic with a  closed subspace of $[0,1]^{\omega}$. Hence (i) implies (ii). That (ii) implies (iii) follows from Proposition \ref{p3.3}. Of course, if (iii) holds, then $(0,1)$ is sequential, so $\mathbb{R}$, being homeomorphic with $(0,1)$, is also sequential.
\end{proof}

In what follows, for a metric space $\mathbf{X}=\langle X,d_X\rangle$ and
for each $n\in\omega$, we put $X_n=X$ and $d_n=d_X$. Then $d$ is the metric
on $\mathbf{X}^{\omega}$ defined by  (\ref{4}). The following lemma is applied
to the proof of our Theorem \ref{t4.6} which is a generalization and an
extension of Theorem 4.55 of \cite{herl} and of Theorem 3.6 of \cite{kw}:

\begin{lemma}
\label{l4.4} If $\mathbf{X}$ is Cantor completely metrizable
second-countable space, then every closed subspace of $\mathbf{X}$ is
separable.
\end{lemma}

\begin{proof} Let us fix a countable base $\mathcal{B}=\{B_n: n\in\omega\}$ of a second-countable Cantor completely metrizable space  $\mathbf{X}$ and a Cantor complete metric $\rho$ inducing the topology of $\mathbf{X}$. Let $A$ be a non-empty closed set in $\mathbf{X}$ and let $N(A)=\{n\in\omega: A\cap B_n\neq\emptyset\}$. Using $\mathcal{B}$, for each $n\in N(A)$, we can define in $\mathbf{ZF}$ a sequence $(A(n,k))_{k\in\mathbb{N}}$ of members of $\mathcal{B}$, such that $\emptyset\neq A\cap A(n,k+1)\subseteq A\cap A(n,k)\subseteq B_n $ and $\delta_{\rho}(A(n,k))<\frac{1}{k}$ for each $k\in\mathbb{N}$. Since $\rho$ is Cantor complete, we obtain a sequence $(a_n)_{n\in\mathbb{N}}$ where $a_n\in A\cap\bigcap_{n\in\mathbb{N}}\text{cl}_{\mathbf{X}}A(n,k)$. The set $\{a_n: n\in\mathbb{N}\}$ is countable and dense in $A$.
\end{proof}

\begin{definition}
\label{d4.5} If $\mathcal{Y}=\{\mathbf{Y}_{j}:j\in J\}$ is a collection of
topological or metric spaces, then we say that $\mathcal{Y}$ \textit{has a
choice function} if the collection $\{Y_{j}:j\in J\}$ of the underlying sets
has a choice function.
\end{definition}

\begin{theorem}
\label{t4.6} For every Cantor complete second-countable metric space $%
\mathbf{Y}=\langle Y, d_Y\rangle$, the following conditions are all
equivalent:

\begin{enumerate}
\item[(i)] $\mathbf{Y}$ is sequential;

\item[(ii)] every non-empty family of non-empty complete metric subspaces of 
$\mathbf{Y}$ has a choice function;

\item[(iii)] every non-empty countable family of non-empty complete metric
subspaces of $\mathbf{Y}$ has a choice function;

\item[(iv)] all complete metric subspaces of $\mathbf{Y}$ are closed in $%
\mathbf{Y}$;

\item[(v)] every sequentially closed subspace of $\mathbf{Y}$ is separable.

\item[(vi)] $\mathbf{s}-\mathbf{Loeb}(\mathbf{Y})$;

\item[(vii)] $\mathbf{CAC}(\mathbf{Y}, scl)$;

\item[(viii)] $\mathbf{PCAC}(\mathbf{Y}, scl)$.
\end{enumerate}
\end{theorem}

\begin{proof} Since every complete subspace of $\mathbf{Y}$ is sequentially closed in $\mathbf{Y}$, while $\mathbf{Y}$ is Loeb by Theorem \ref{t2.2}, we infer that conditions (ii), (iii), (vi), (vii) and (viii) follow from (i). It is obvious that (iii) follows from (ii). To prove that (iv) follows from (iii), we assume that $\mathbf{Z}$ is a complete metric subspace of $\mathbf{Y}$. Suppose that $y\in\text{cl}_{\mathbf{Y}}(Z)\setminus Z$. For each $n\in\mathbb{N}$, the metric subspace of $\mathbf{Y}$ whose underlying set is $\overline{B}_{d_Y}(y, \frac{1}{n})\cap Z$ is complete and non-empty. Assuming that (iii) holds, we fix $f\in\prod_{n\in\mathbb{N}}(\overline{B}_{d_Y}(y, \frac{1}{n})\cap Z)$. The sequence $(f(n))_{n\in\mathbb{N}}$ is a sequence of points of $Z$ which converges in $\mathbf{Y}$ to $y$. Since $\mathbf{Z}$ is complete, we obtain that $y\in Z$, so $Z$ is closed in $\mathbf{Y}$. Hence (iii) implies (iv). Now, let $\mathbf{A}$ be a sequentially closed metric subspace of $\mathbf{Y}$. Then $\mathbf{A}$ is a complete metric subspace of $\mathbf{Y}$. If (iv) holds, then the set $A$ is closed in $\mathbf{Y}$ and, in view of Lemma \ref{l4.4}, $\mathbf{A}$ is separable. Hence (iv) implies (v). It is obvious that (v) implies (i). It is also obvious that (vi) implies (vii) and (vii) implies (viii). To prove that (viii) implies (i), we assume (viii) and consider a sequentially closed subset $A$ of $\mathbf{Y}$. We show that $A$ is closed. Fix $x\in\text{cl}_{\mathbf{Y}}(A)$. For each $n\in\mathbb{N}$, let $G_{n}=\overline{B}_{d_y}(x,\frac{1}{n})\cap A$. Clearly, $\mathcal{G}=\{G_n: n\in\mathbb{N}\}$ is a family of non-empty sequentially closed subsets of $\mathbf{Y}$. Any partial choice function of $\mathcal{G}$ is a sequence of points of $A$ converging to $x$. Thus, $x\in A$, so $A$ is closed in $\mathbf{Y}$ as required.
\end{proof}

The following corollaries follow directly from Theorems \ref{t2.2} and \ref%
{t4.6}:

\begin{corollary}
\label{c4.7} For every Cantor complete second-countable metric space $%
\mathbf{X}=\langle X, d_X\rangle$ and for $\mathbf{Y}=\mathbf{X}^{\omega}$,
conditions (i)-(viii) of Theorem \ref{t4.6} are all equivalent.
\end{corollary}

\begin{corollary}
\label{c4.8} (a) For every $m\in\mathbb{N}\cup\{\omega\}$, if $\mathbf{Y}%
=[0, 1]^{m}$, then conditions (i)-(viii) of Theorem \ref{t4.6} are all
equivalent with $\mathbf{S}(\mathbb{R})$.\newline
(b) For every $m\in\mathbb{N}\cup\{\omega\}$, if $\mathbf{Y}=\mathbb{R}^{m}$%
, then conditions (i)-(viii) of Theorem \ref{t4.6} are all equivalent.
\end{corollary}

\begin{corollary}
\label{c4.9} In Cohen's basic model $\mathcal{M}$1 of \cite{hr}, $\mathbf{%
s-Loeb}(\mathbb{R})$ fails.
\end{corollary}

\begin{proof} It is known that, in the model $\mathcal{M}$1 of \cite{hr}, the set $A$ of all added Cohen reals is sequentially closed but not closed in $\mathbb{R}$. Hence $\mathbb{R}$ is not sequential in $\mathcal{M}$1 and, therefore, in view of Corollary \ref{c4.8}, $\mathbf{s-Loeb}(\mathbb{R})$ fails in $\mathcal{M}$1.
\end{proof}

\begin{remark}
\label{r4.10} Since $\mathbb{R}$ is Loeb, it follows from Corollary \ref%
{c4.9} that, in a model $\mathbf{ZF}$, it may happen for a metrizable space $%
\mathbf{X}$ that $\mathbf{L}(\mathbf{X})$ holds and $\mathbf{s-Loeb}(\mathbf{%
X})$ fails. Clearly, $\mathbf{s-Loeb}(\mathbf{X})$ implies $\mathbf{L}(%
\mathbf{X})$ for every space $\mathbf{X}$, so $\mathbf{s-Loeb}(\mathbf{X})$
is stronger than $\mathbf{L}(\mathbf{X})$.
\end{remark}

\begin{remark}
\label{r4.11} Regarding conditions (i)-(viii) of Theorem \ref{t4.6}, if we
drop the requirement that the space $\mathbf{Y}$ be second-countable, then
one may wonder which of the conditions (ii)-(viii) are equivalent with (i).
Of course, if (v) holds, then $\mathbf{Y}$ is second-countable. We notice
that the implication (vii)$\to$(viii) of Theorem \ref{t4.6} holds for every
space $\mathbf{Y}$. Theorem \ref{t4.14} given below shows that, in a model
of $\mathbf{ZF}$, the implication (viii)$\to$(vii) of Theorem \ref{t4.6}
need not hold even for a compact sequential metrizable space $\mathbf{Y}$.
Now, let us assume that $\mathbf{Y}=\langle Y, d_Y\rangle$ is a complete
metric space. We can easily observe that a set $A\subseteq Y$ is
sequentially closed in $\mathbf{Y}$ if and only if the metric subspace $%
\mathbf{A}$ of $\mathbf{Y}$ is complete. Therefore, conditions (i) and (iv)
of Theorem \ref{t4.6} are equivalent for every complete metric space $%
\mathbf{Y}$. Assuming that $\mathbf{Y}$ is Cantor complete, in much the same
way, as in the proof to Theorem \ref{t4.6} that (viii) implies (i), we can
show that each of conditions (ii), (iii), (vi), (vii) and (viii) implies
that $\mathbf{Y}$ is sequential.
\end{remark}

\begin{theorem}
\label{t4.12} The following conditions are all equivalent:

\begin{enumerate}
\item[(i)] $\mathbf{AC}$;

\item[(ii)] for every Cantor complete sequential metric space $\mathbf{Y}$,
every non-empty family of non-empty complete metric subspaces of $\mathbf{Y}$
has a choice function;

\item[(iii)] for every Cantor complete metric space $\mathbf{Y}$, $\mathbf{s}%
-\mathbf{Loeb}(\mathbf{Y})$ holds.
\end{enumerate}
\end{theorem}

\begin{proof}
Of course, $\mathbf{AC}$ implies (ii) and (iii). Suppose that $\mathbf{AC}$ does not hold. Let $\mathcal{A}=\{A_j:  j\in J\}$ be a collection of non-empty pairwise disjoint sets which does not have a choice function. Let $Y=\bigcup_{j\in J}A_j$ and let $d_Y$ be the discrete metric on $Y$. Then the metric space $\mathbf{Y}=\langle Y, d_Y\rangle$ is Cantor complete and sequential, while $\mathbf{s-Loeb}(\mathbf{Y})$ fails because the collection $\{\mathbf{A}_j: j\in J\}$ of complete and sequentially closed metric subspaces of $\mathbf{Y}$ does not have a choice function.
\end{proof}

Using similar arguments as above, together with the fact that $\mathbf{CAC}$
and $\mathbf{PCAC}$ are equivalent (see Theorem 2.12 in \cite{herl} and Form
8 in \cite{hr}), one can prove the following theorem:

\begin{theorem}
\label{t4.13} The following conditions are all equivalent:

\begin{enumerate}
\item[(i)] $\mathbf{CAC}$;

\item[(ii)] for every Cantor complete sequential metric space $\mathbf{Y}$,
every non-empty countable family of non-empty complete metric subspaces of $%
\mathbf{Y}$ has a choice function;

\item[(iii)] for every Cantor complete metric space $\mathbf{Y}$, $\mathbf{%
CAC}(\mathbf{Y}, scl)$ holds;

\item[(iv)] for every Cantor complete metric space $\mathbf{Y}$, $\mathbf{%
PCAC}(\mathbf{Y}, scl)$ holds.
\end{enumerate}
\end{theorem}

We recall that a set $A$ is \textit{weakly Dedekind-finite} if the power set $\mathcal{P}(A)$ is Dedekind-finite.

\begin{theorem}
\label{t4.14} (a) It is relatively consistent with $\mathbf{ZF}$ the
existence of a sequential countably compact metrizable space $\mathbf{Y} $
which does not satisfy $\mathbf{s-Loeb}(\mathbf{Y})$.\newline
(b) It is independent of $\mathbf{ZF}$ that every simultaneously sequential,
compact and metrizable space $\mathbf{Y}$ satisfies $\mathbf{CAC}(\mathbf{Y}%
,scl)$. In particular, in some models of $\mathbf{ZF}$, the following
statement is false: For every compact metrizable space $\mathbf{Y}$, if $%
\mathbf{PCAC}(\mathbf{Y},scl)$ holds, then $\mathbf{CAC}(\mathbf{Y},scl)$
holds.
\end{theorem}

\begin{proof} 
(a) \ Let $\mathcal{M}$ be a model of $\mathbf{ZF}$ such that there exists in $\mathcal{M}$ an infinite weakly
Dedekind-finite set $Y$ (it is known that, in the Basic Fraenkel Model $%
\mathcal{N}$1 of \cite{hr}, the set of atoms is weakly Dedekind-finite and
this result is transferable in $\mathbf{ZF}$).  Clearly, the discrete space $\mathbf{Y}=\langle Y, \mathcal{P}(Y)\rangle$ is a sequential countably compact metrizable space such that every subset
of $Y$ is closed. However, $\mathbf{Y}$ has no Loeb function as, otherwise,
by Remark \ref{r2.5}, $Y$ would be well-orderable, hence Dedekind-infinite.

(b) It was proved in \cite{hsh} that, for every prime number $p$, there exists a model $\mathcal{M}$ of $\mathbf{ZF}$ in which every infinitely countable collection of $p$-element sets has a partial choice function but there exists in $\mathcal{M}$ a countable collection of $p$-element sets without a choice function.  Let us fix a  model $\mathcal{M}$ in which every infinite countable family of 2-element sets has a partial choice function but, in $\mathcal{M}$, there exists a family $\mathcal{A}=\{ A_n: n\in\mathbb{N}\}$ of pairwise disjoint $2$-element sets which does not have a choice function. Let us work in $\mathcal{M}$ with such a family $\mathcal{A}$. Take a point $\infty\notin\bigcup\mathcal{A}$ and put $%
Y= \{\infty \}\cup\bigcup\mathcal{A}$. We define a metric $d$ on $Y$ by
requiring: 
\begin{equation*}
d(x,y)=d(y,x)=\left\{ 
\begin{array}{c}
0,\text{ if }x=y \\ 
\frac{1}{n},\text{ if }x\in A_{n}\text{, }y\in A_{m}\text{ and }n\leq m \\ 
\frac{1}{n}\text{ if }x\in A_{n}\text{ and }y=\infty 
\end{array}%
\right. .
\end{equation*}%
Clearly, $\mathbf{Y}=\langle Y, d\rangle$ is a compact metric space. We prove that $\mathbf{Y}$ is sequential. Let $G$ be a sequentially closed subset of $\mathbf{Y}$.  Suppose that $G$ is not closed in $\mathbf{Y}$. Then $\infty\in\text{cl}_{\mathbf{Y}}(G)\setminus G$; thus, the set $N(G)=\{ n\in\omega: G\cap A_n\neq\emptyset\}$ is infinite. Since $\{ A_n: n\in N(G)\}$ has a partial choice function, there exist an infinite set $N\subseteq N(G)$ and a point $z\in\prod_{n\in N}(G\cap A_n)$. Then $(z(n))_{n\in N}$ is a sequence of points of $G$ which converges to $\infty$. This is impossible because $G$ is sequentially closed, while $\infty\notin G$. It follows from the contradiction obtained that space $\mathbf{Y}$ is sequential. Hence, by Theorem \ref{t3.15}, $\mathbf{PCAC}(\mathbf{Y},scl)$ is satisfied. However, $\mathbf{CAC}(%
\mathbf{Y},scl)$ fails because $\{A_{n}:n\in \omega \}$ is a family of
closed sets of $\mathbf{Y}$ without a choice function.
\end{proof}

\begin{theorem}
\label{t4.15} If, for every Cantor completely metrizable sequential,
sequentially compact space $\mathbf{\ Y}$, it holds true that every
sequentially closed subspace of $\mathbf{Y }$ has a well-orderable dense
subset, then $\mathbf{IWDI}$ holds.
\end{theorem}

\begin{proof}
Let us assume that $\mathbf{IWDI}$ does not hold and fix an
infinite set $A$ such that $\mathcal{P}(A)$ is Dedekind-finite. We apply the hedgehog-like scheme whose general description is given in Example 4.13 of \cite{PW} (see also Example 4.1.5 of \cite{En}). Namely, for every $a\in A$, let $%
X_{a}=\{0\}\cup\{\langle a,\frac{1}{n}\rangle: n\in \mathbb{N}\}$. Put $X=\bigcup
\{X_{a}:a\in A\}$ and define a metric $d$ on $X$ by requiring, for all $x,y\in X$, $a,b\in A$ and $n,m\in\mathbb{N}$, the following: 
\begin{equation*}
d(x,y)=d(y,x)=\left\{ 
\begin{array}{c}
0 \text{ if }x=y \\ 
\frac{1}{n} \text{ if }x=0\text{ and }y=\langle a,\frac{1}{n}\rangle\\ 
\frac{1}{n}+\frac{1}{m}\text{ if }x=\langle a,\frac{1}{m}\rangle, y=\langle b,\frac{1%
}{n}\rangle\text{ where }a\neq b \\
\vert \frac{1}{n}-\frac{1}{m}\vert \text{ if } x=\langle a, \frac{1}{n}\rangle, y=\langle a, \frac{1}{m}\rangle 
\end{array}%
\right. .
\end{equation*}%
We show that the metric space $\mathbf{X}=\langle X, d\rangle$ is sequential. To this aim, suppose that $\mathbf{X}$ is not sequential. Fix a sequentially closed, not closed set $G$ of $\mathbf{X}$ and a point $x\in \text{cl}_{\mathbf{X}}(G)\setminus G$. Then $x=0$ because $0$ is the unique accumulation point of $\mathbf{X}$. If $a\in A$ is such that $G\cap X_{a}$ is infinite, then $G\cap X_{a}$ with the order it inherits from the obvious order of $X_{a}$ is a sequence of points of $X$ converging to $0$ and, in consequence, $0\in G$ - a contradiction. Hence $G\cap X_{a}$ is finite for each $a\in A$. Since $G$ is infinite, it follows that there exist
infinitely many $a\in A$ with $G\cap X_{a}\neq \emptyset $. We define an
equivalence relation $\sim $ on $A$ by letting, for all $a,b\in A$, the following:
$$a\sim b \Leftrightarrow  \{n\in\mathbb{N}: \langle a, \frac{1}{n}\rangle\in G\}=\{n\in\mathbb{N}: \langle b, \frac{1}{n}\rangle\in G\}.$$
It is easily seen that if $A/\sim $ were finite, then $0\notin\text{cl}_{\mathbf{X}}(G)$ - a contradiction. Hence $A/\sim $ is infinite. Since the set $[\mathbb{N}]^{<\mathbb{N}}$ of all finite subsets
of $\mathbb{N}$ is countable, it follows that $A/\sim $ is an infinite countable
partition of $A$ into non-empty sets. This contradicts the fact that $\mathcal{P}(A)$ is a
Dedekind-finite set. We infer from the contradiction obtained that $\mathbf{X}$ is sequential. One can easily check that $\mathbf{X}$ is Cantor complete. However, $\mathbf{X}$ has no dense well-orderable subset as,
otherwise, we could define an infinite well-orderable subset of $A$. To complete the proof, it suffices to prove that $\mathbf{X}$ is sequentially compact.  Indeed, given  sequence $(x_k)_{k\in\omega}$  of points of $X\setminus\{0\}$, the assumption that $\mathcal{P}(A)$ is Dedekind-finite implies that the set $H=\{a\in A: {\exists}_{k\in\omega} x_k\in X_{a}\setminus\{0\}\}$ is finite. This, together with the sequential compactness in $\mathbf{X}$ of the set $\bigcup_{a\in H}X_{a}$, implies that $\mathbf{X}$ is sequentially compact.  
\end{proof}

\begin{remark}
\label{r4.16} Let us notice that the hedgehog-like space $\mathbf{X}$ from
the proof to Theorem \ref{t4.15} is countably compact; thus, in view of
Theorem \ref{t3.16}, $\mathbf{X}\times\mathbf{Y}$ is sequential for every
sequential space $\mathbf{Y}$.
\end{remark}

Let us discuss in brief the following open problem posed in Question 3.9 of 
\cite{kw}:

\begin{problem}
\label{prob4.17} Must $\mathbb{R}^{\omega}$ be sequential in every model of $%
\mathbf{ZF}$ in which $\mathbb{R}$ is sequential?
\end{problem}

If $\mathbb{R}$ is sequential, then, by Theorem \ref{t4.2} and Proposition %
\ref{p3.3}, every locally compact subspace of the Hilbert cube $%
[0,1]^{\omega }$ is sequential. However, $(0,1)^{\omega }$ is not locally
compact. This is partly why we are unable to give a solution to Problem \ref%
{prob4.17} in this paper. Of course, by virtue of Proposition \ref{p3.3}, if 
$\mathbb{R}^{\omega }$ is sequential, so is $\mathbb{N}^{\omega }$.
Unfortunately, we do not know how to solve the following problem:

\begin{problem}
\label{prob4.18} Must $\mathbb{N}^{\omega}$ be sequential in every model of $%
\mathbf{ZF}$ in which $\mathbb{R}$ is sequential?
\end{problem}

\begin{proposition}
\label{p4.19} If $\mathbb{N}^{\omega}$ is sequential, then $\mathbb{R}$ is
sequential. Hence, $\mathbb{N}^{\omega}$ is not sequential in Cohen's basic
model $\mathcal{M}$1 of \cite{hr}.
\end{proposition}

\begin{proof} Suppose that $\mathbb{N}^{\omega}$ is sequential. Then the Cantor cube $\{0,1\}^{\omega}$ is also sequential, so $\mathbb{R}$ is sequential by Theorem \ref{t4.2}. Now, the second part of the proposition follows from the first part and from Corollaries \ref{c4.8}-\ref{c4.9}.
\end{proof}

\begin{remark}
\label{r4.20} To prove that if $\mathbb{N}^{\omega}$ is sequential, then $%
\mathbb{R}$ is sequential, one can use the following argument. Assuming that 
$\mathbb{N}^{\omega}$ is sequential, we deduce that the subspace $\mathbb{R}%
\setminus\mathbb{Q}$ of $\mathbb{R}$ is sequential. Therefore, if $A$ is a
sequentially closed subset of $\mathbb{R}$, then, since $B=A\cap(\mathbb{R}%
\setminus\mathbb{Q})$ is sequentially closed in $\mathbb{R}\setminus\mathbb{Q%
}$, it follows from Theorem \ref{t4.6} that there exists a countable subset $%
D$ of $C$ such that $C\subseteq \text{cl}_{\mathbb{R}\setminus\mathbb{Q}}(D)$%
. The set $E=D\cup(A\cap\mathbb{Q})$ is countable and $A\subseteq\text{cl}_{%
\mathbb{R}}(E)$, so $\mathbb{R}$ is sequential by Theorem \ref{t4.6}.
\end{remark}

It is interesting to compare $\omega-\mathbf{CAC}(\mathbb{R})$ with
conditions (ii)-(iii) of the following theorem which leads to partial
solutions of Problems \ref{prob4.17} and \ref{prob4.18}:

\begin{theorem}
\label{t4.21} Suppose that $\mathbf{X}$ is a Loeb regular space. If $\mathbf{%
X}$ is first-countable, then the following conditions are equivalent:

\begin{enumerate}
\item[(i)] $\mathbf{X}$ is sequential;

\item[(ii)] for every family $\{A_j: j\in J\}$ of non-empty sequentially
closed subsets of $\mathbf{X}$, there exists a family $\{B_j: j\in J\}$ of
non-empty well-orderable sets such that $B_j\subseteq A_j$ for each $j\in J$;

\item[(iii)] for every family $\{A_n: n\in\omega\}$ of non-empty
sequentially closed subsets of $\mathbf{X}$, there exist an infinite subset $%
N$ of $\omega$ and a family $\{B_n: n\in N\}$ of non-empty well-orderable
sets such that $B_n\subseteq A_n$ for each $n\in N$.
\end{enumerate}

Moreover, if $\mathbf{X}$ is second-countable, then conditions (i)-(iii) are
equivalent with the following:

\begin{enumerate}
\item[(iv)] for every family $\{A_n: n\in\omega\}$ of non-empty sequentially
closed subsets of $\mathbf{X}$, there exists a family $\{ D_n: n\in\omega\}$
of countable sets such that $D_n\subseteq A_n\subseteq\text{cl}_{\mathbf{X}%
}(D_n)$ for each $n\in\omega$ and, moreover, the set $D=\bigcup_{n\in\omega}
D_n$ is countable.
\end{enumerate}
\end{theorem}

\begin{proof}  Let  $f$ be a fixed Loeb function of $\mathbf{X}$. It is obvious that (ii) implies (iii). To prove that (i) implies (ii), we consider a family  $\{A_j: j\in J\}$ of non-empty sequentially closed subsets of $\mathbf{X}$ and we assume (i). Then, for each $j\in J$, the set $A_j$ is closed in $\mathbf{X}$, so we can put $B_j=\{f(A_j)\}$ to obtain a collection $\{B_j: j\in J\}$ of one-point (so well-orderable) sets such that $B_j\subseteq A_j$ for each $j\in J$. This is why (i) implies (ii).

Let us assume (iii) and let $A$ be a sequentially closed subset of $\mathbf{X}$. Assuming that $\mathbf{X}$ is first-countable, we show that $A$ is closed in $\mathbf{X}$. Let $x\in\text{cl}_{\mathbf{X}}(A)$. Fix any countable base $\{U_n: n\in\omega\}$ of neighbourhoods of $x$ such that $U_{n+1}\subseteq U_n$ for each $n\in\omega$. Put $A_n=A\cap\text{cl}_{\mathbf{X}}(U_n)$ for each $n\in\omega$. Then $\{A_n: n\in\omega\}$ is a family of non-empty sequentially closed subsets of $\mathbf{X}$. By (iii), there exist an infinite subset $N$ of $\omega$ and a collection $\{B_n: n\in N\}$ of well-orderable sets such that $B_n\subseteq A_n$ for each $n\in N$. Since the sets $B_n$ are well-orderable and $\mathbf{X}$ is first-countable,  while $A_n$ are sequentially closed in $\mathbf{X}$, the inclusion $\text{cl}_{\mathbf{X}}(B_n)\subseteq A_n$ holds for each $n\in N$. We notice that, since $\mathbf{X}$ is regular, $(f(\text{cl}_{\mathbf{X}}(B_n)))_{n\in N}$ is a sequence of points of $A$ which converges to $x$, so $x\in A$ and, in consequence, $A$ is closed in $\mathbf{X}$. Hence (iii) implies (i). If $\mathbf{X}$ is first-countable, then (iv) implies (i) because (iv) implies (iii).  

Finally, we assume that the space $\mathbf{X}$ is a second-countable and sequential. We fix a countable base $\mathcal{B}=\{V_j:  j\in\omega\}$ of $\mathbf{X}$. We consider any family $\{A_n: n\in\omega\}$ of non-empty sequentially closed subsets of $\mathbf{X}$. Since $\mathbf{X}$ is sequential, the set $A_n$ is closed in $\mathbf{X}$ for each $n\in\omega$. Let $M_n=\{j\in\omega: A_n\cap V_j\neq\emptyset\}$ and let $D_n=\{f(A_n\cap\text{cl}_{\mathbf{X}}(V_j)): j\in M_n\}$ for each $n\in\omega$. Obviously, $D_n\subseteq A_n\subseteq\text{cl}_{\mathbf{X}}(D_n)$ for each $n\in\omega$. Since we can fix a sequence $(\psi)_{n\in\omega}$ of injections $\psi_n: D_n\to\omega$, the set $D=\bigcup_{n\in\omega}D_n$ is countable.
\end{proof}

One can easily deduce the following corollary from Theorems \ref{t2.2} and %
\ref{t4.21}:

\begin{corollary}
\label{c4.22} For every Cantor completely metrizable second-countable space $%
\mathbf{X}$, conditions (i)-(iv) of Theorem \ref{t4.21} are equivalent. In
particular, conditions (i)-(iv) of Theorem \ref{t4.21} are equivalent if $%
\mathbf{X}$ is one of the following spaces: $\mathbb{R}^m$ where $m\in%
\mathbb{N}$, $\mathbb{R}^{\omega}$, $\mathbb{N}^{\omega}$, $[0, 1]^{\omega}$%
, $\{0,1\}^{\omega}$.
\end{corollary}

It is well known that the sets $\mathbb{R}$ and $\mathbb{R}\times \mathbb{R}$
are equipotent in $\mathbf{ZF}$; however, we cannot locate in the literature
a proof to the following useful proposition, so we include a sketch of a
proof to it:

\begin{proposition}
\label{p4.23}  The sets $\mathbb{R}^{\omega}$ and $\mathbb{R}$ are
equipotent in $\mathbf{ZF}$.
\end{proposition}

\begin{proof} Clearly, $\vert \mathbb{R}\vert\leq\vert\mathbb{R}^{\omega}\vert$  because $\mathbb{R}$ is equipotent with a subset of $\mathbb{R}^{\omega}$. To show that $\vert \mathbb{R}^{\omega}\vert\leq\vert\mathbb{R}\vert$, in much the same way, as in the proof to Theorem \ref{t2.2}, we notice that, considering $\mathbb{R}$ with its natural topology, we obtain that the metrizable space $\mathbb{R}^{\omega}$ is second-countable. We fix a countable base $\mathcal{B}$ of $\mathbb{R}^{\omega}$ and define an injection $H:\mathbb{R}^{\omega}\to\mathcal{P}(\mathcal{B})$ as follows: $H(x)=\{B\in\mathcal{B}: x\in B\}$ for each $x\in\mathbb{R}^{\omega}$. This shows that $\vert\mathbb{R}^{\omega}\vert\leq\vert\mathcal{P}(\mathcal{B})\vert$. Since the sets $\mathbb{R}$ and $\mathcal{P}(\mathcal{B})$ are equipotent, we deduce that $\vert\mathbb{R}^{\omega}\vert\leq\vert\mathbb{R}\vert$. As the Cantor-Schr\"oder-Bernstein theorem is provable in $\mathbf{ZF}$, we can infer that $\vert\mathbb{R}^{\omega}\vert=\vert\mathbb{R}\vert$.
\end{proof}

We can extend Proposition 4.57 of \cite{herl} by the following:

\begin{theorem}
\label{t4.24} $\omega-\mathbf{CAC}(\mathbb{R})$ implies $\mathbb{R}^{\omega}$
is sequential.
\end{theorem}

\begin{proof} Assume $\omega-\mathbf{CAC}(\mathbb{R})$.  By Corollary \ref{c2.3}, the space $\mathbb{R}^{\omega}$ is Loeb. Of course, $\mathbb{R}^{\omega}$ is a second-countable regular space.  We check that condition (ii) of Theorem \ref{t4.21} is satisfied for $\mathbf{X}=\mathbb{R}^{\omega}$. Let $\{A_n: n\in\omega\}$ be a collection of  non-empty subsets of $\mathbb{R}^{\omega}$. By  Proposition \ref{p4.23}, there exists a bijection $h:\mathbb{R}\to\mathbb{R}^{\omega}$. It follows from $\omega-\mathbf{CAC}(\mathbb{R})$ that there exists a collection $\{C_n: n\in\omega\}$ of non-empty at most countable sets such that $C_n\subseteq h^{-1}(A_n)$ for each $n\in\omega$. Put $B_n= h(C_n)$ for each $n\in\omega$. Clearly, all the sets $B_n$ are at most countable, so well-orderable. Moreover, $\emptyset\neq B_n\subseteq A_n$ for each $n\in\omega$. This, together with Theorem \ref{t4.21}, implies that $\mathbb{R}^{\omega}$ is sequential. 
\end{proof}

\begin{corollary}
\label{c4.25} It is relatively consistent with $\mathbf{ZF}$ that
simultaneously $\mathbb{R}^{\omega}$ is sequential and $\mathbf{CAC}(\mathbb{%
R})$ fails.
\end{corollary}

\begin{proof} It is known that, in Feferman-Levy model $\mathcal{M}$9 of \cite{hr}, $\omega-\mathbf{CAC}(\mathbb{R})$ holds and $\mathbf{CAC}(\mathbb{R})$ fails. In view of Theorem \ref{t4.24}, $\mathbb{R}^{\omega}$ is sequential in $\mathcal{M}$9.
\end{proof}

Since spaces homeomorphic with closed subspaces of sequential spaces are
sequential, we can deduce the following corollary from Theorem \ref{t4.24}:

\begin{corollary}
\label{c4.26} $\omega-\mathbf{CAC}(\mathbb{R})$ implies that the space $%
\mathbb{N}^{\omega}$ is sequential and, for every $n\in\mathbb{N}$, the
space $\mathbb{R}^{n}$ is sequential.
\end{corollary}

\bigskip

\noindent \textsc{Kyriakos Keremedis}\newline
Department of Mathematics, University of the Aegean, Karlovassi, Samos
83200, Greece.\newline
\emph{E-mail}: kker@aegean.gr \bigskip

\noindent\textsc{Eliza Wajch}\newline
Institute of Mathematics and Physics,  University of Natural Sciences and
Humanities in Siedlce, ul. 3 Maja 54, 08-110 Siedlce, Poland.\newline
\emph{E-mail}: eliza.wajch@wp.pl


\begin{thebibliography}{99}
\bibitem{Bo} T. K. Boehme, \textit{Linear $s$-spaces}, Proc. Symp.
Convergence Structures, Univ. of Oklahoma, 1965.

\bibitem{En} R. Engelking, \textit{General Topology}, Sigma Series in Pure
Mathematics 6, Heldermann, Berlin 1989.

\bibitem{F1} S. P. Franklin, \textit{Spaces in which sequences suffice},
Fund. Math. 57 (1965), 107--115.

\bibitem{F2} S. P. Franklin, \textit{Spaces in which sequences suffice II},
Fund. Math. 61 (1967), 51--56.

\bibitem{gg2} G. Gutierres, \textit{On countable choice and sequential spaces%
}, Math. Logic Quart. \textbf{54}, No. 2, (2008), 145--152.

\bibitem{hsh} E. J. Hall, S. Shelah, \textit{Partial choice functions for
families of finite sets}, Fund. Math. 220 (2013) 207--216.

\bibitem{herl} H. Herrlich, \textit{Axiom of Choice}, Lecture Notes Math. 
\textbf{1876}, Springer, New York, 2006.

\bibitem{hkt} H. Herrlich, K. Keremedis, E. Tachtsis, \textit{Countable sums
and products of Loeb and selective metric spaces}, Comment. Math. Univ.
Carolin. \textbf{46}, No. 2, (2005) 373--384.

\bibitem{hr} P. Howard and J. E. Rubin, \textit{Consequences of the axiom of
choice, }Math. Surveys and Monographs \textbf{59 }A.M.S. Providence R.I.,
1998.

\bibitem{kyr} K. Keremedis, \textit{On sequentially compact and related
notions of compactness of metric spaces in} $\mathbf{ZF}$, Bull. Polish
Acad. Sci. Math. \textbf{64} (2016) 29--46.

\bibitem{kyr2} K. Keremedis, \textit{On the relative strength of forms of
compactness of metric spaces and their countable productivity in $\mathbf{ZF}
$}, Topology Appl. \textbf{159} (2012) 3396--3403.

\bibitem{kyr1} K. Keremedis, \textit{Compact and Loeb Hausdorff spaces and
the axiom of choice for families of finite sets}, Math. Logic Quart. \textbf{%
58}, No. 3, (2012) 130--138.

\bibitem{kt} K. Keremedis and E. Tachtsis: Compact metric spaces and weak
forms of the axiom of choice, Math. Logic Quart. \textbf{47} (2001) 117--128.

\bibitem{kertach} K. Keremedis and E. Tachtsis, \textit{On Loeb and weakly
Loeb Hausdorff spaces}, Sci. Math. Jpn. \textbf{53} (2001) 247--251.

\bibitem{keta} K. Keremedis and E. Tachtsis, \textit{Weak axioms of choice
for metric spaces}, Proc. Amer. Math. Soc. \textbf{133}, No. 12, (2005)
3691--3701.

\bibitem{kw} K. Keremedis, Eliza Wajch, \textit{On densely complete metric
spaces and extensions of uniformly continuous mappings in} $\mathbf{ZF}$,
preprint available at arxiv:1901.08709.

\bibitem{ku} K. Kunen, \textit{The Foundations of Mathematics}, Individual
Authors and College Publications, London 2009.

\bibitem{loeb} P. A. Loeb, \textit{A new proof of the Tychonoff theorem},
Amer. Math. Monthly \textbf{72} (1965) 711-717.

\bibitem{mich} E. Michael, \textit{Local compactness and Cartesian products
of quotient maps and $K$-spaces}, Ann. Inst. Fourier 18(2) (1968) 281--286.

\bibitem{mm} M. Morillon, \textit{Notions of compactness for some special
subsets of} $\mathbb{R}^{I}$ \textit{and some weak forms of the axiom of
choice}, J. Symb. Logic \textbf{75} (2010) 255--268.

\bibitem{OPWZ} C. \"Ozel, A. Pi\k{e}kosz, E. Wajch and H. Zekraoui, \textit{%
\ The minimizing vector theorem in symmetrized max-plus algebra}, J. Convex
Anal. (May 2019), available online.

\bibitem{PW} A. Pi\k{e}kosz and E. Wajch, \textit{Quasi-metrizability of
bornological biuniverses in ZF}, J. Convex Anal. 22, No. 4, (2015)
1041--1060.

\bibitem{Wh} J. H. C. Whitehead, \textit{\ A note on a theorem of Borsuk}, Bull. Amer. Math. Soc. 54 (1948) 1125--1132.
\end{thebibliography}
\end{document}